\documentclass[onefignum,onetabnum]{siamart190516}

\usepackage{lipsum}
\usepackage{amsfonts}
\usepackage{graphicx}
\usepackage{epstopdf}
\usepackage{algorithmic}
\ifpdf
  \DeclareGraphicsExtensions{.eps,.pdf,.png,.jpg}
\else
  \DeclareGraphicsExtensions{.eps}
\fi

\newsiamremark{remark}{Remark}
\newsiamremark{hypothesis}{Hypothesis}
\crefname{hypothesis}{Hypothesis}{Hypotheses}
\newsiamthm{claim}{Claim}

\title{Event-triggered gain scheduling  of  reaction-diffusion
PDEs}

\author{ Iasson Karafyllis\thanks{Department of Mathematics, National Technical University of Athens, Zografou Campus, 15780, Athens, Greece 
  (\email{iasonkar@central.ntua.gr,iasonkaraf@gmail.com}).} \and Nicol\'as Espitia\thanks{Univ. Lille, CNRS, Centrale Lille, UMR 9189 - CRIStAL - Centre de Recherche en Informatique Signal et Automatique de Lille, F-59000 Lille, France. 
  (\email{nicolas.espitia-hoyos@univ-lille.fr}).}
\and Miroslav Krstic \thanks{Department of Mechanical and Aerospace Engineering, University of California, San Diego, La Jolla, CA 92093-0411, USA 
  (\email{krstic@ucsd.edu}).}}

\usepackage{amsopn}

\usepackage[utf8]{inputenc}
\ifpdf
\hypersetup{
  pdftitle={Event-triggered gain scheduling of a reaction-diffusion PDEs},
  pdfauthor={N. Espitia, I. Karafyllis, and M. Krstic}
}
\fi

\usepackage{subfigure} 
\newsiamthm{assumption}{Assumption}

\begin{document}

\maketitle
\begin{abstract}
 This paper deals with the problem of boundary stabilization of  1D  reaction-diffusion
PDEs with a time- and space- varying reaction coefficient. The boundary control design relies on the backstepping approach.  The gains of the boundary control are scheduled under two suitable event-triggered mechanisms. More precisely,  gains are computed/updated on events according to  two state-dependent event-triggering conditions: static-based and dynamic-based conditions,  under  which,   the Zeno behavior is avoided and   well-posedness  as well as  exponential stability of the closed-loop system are guaranteed.  Numerical simulations are presented to illustrate the results.
\end{abstract}



\section{Introduction}

Control design of complex systems modeled by partial differential equations (PDEs)  has become a central research area. The  two  traditional ways to act on those complex systems are the \textit{in-domain control} and \textit{boundary control}. For boundary control, \textit{the backstepping method}   has been used as standard tool for designing feedback laws. It has initially emerged to deal with  1D  reaction-diffusion parabolic PDEs in  \cite{Boskovik_kristick_2001_boun_unstable_heat},   \cite{Smyshlyaev-Krstic2004}  and  since then,  the method has been employed to deal with the boundary stabilization of broader classes of PDEs (for an overview  see \cite{krstic2008boundary} and \cite{Meurer2013}). 
One of the most remarkable features of the backstepping approach is the possibility to obtain closed-form analytical solutions for the  kernels of the underlying integral Volterra transformation. Having explicit expressions for the kernels and for  controllers makes implementation simpler and more precise. For instance,  for reaction diffusion systems  with constant parameters,  closed-form solutions for the kernels has been  obtained in terms of  special functions such as the modified Bessel function  \cite{Smyshlyaev-Krstic2004}. When having a time-varying reaction coefficient, a closed-form solution can be obtained through power series for exponential stabilization  \cite{SMYSHLYAEV20051601},  or, in the context of fixed-time stability,   closed-form of time-varying  kernels can been obtained using special functions as   in  \cite{Espitia2018FTSAutomatica}.

Nevertheless, for general reaction diffusion PDEs, obtaining closed-form solutions for the kernels is in general very hard.  For this class of PDEs with  time-and space varying coefficients,  the problem of  boundary stabilization has been very challenging. Time-and space varying reaction coefficients come into play in some applications  such as in trajectory planing and multi-agent systems (see e.g. \cite{MEURER20091182,Freudenthaler2017}), to mention a few. In general, the resulting kernel-PDE is of the form of  hyperbolic spatial operator and first order derivative with respect to time since  the kernel of the Volterra transformation  has to be   time-varying. This brings additional source of complexity to the problem that requires a careful well-posedness analysis and   numerical methods for the solvability.  The solution of the kernel, indeed,  is needed to be  found numerically by  e.g. the method of integral operators  or the so called method of successive approximations  (which  traces back to the seminal work  \cite{Colton1977181}). In this line,  some  contributions have rigorously handled  the well-posedness of time-varying kernels solutions where the reaction term is time and space dependent  as in \cite{MEURER20091182,VasquezTrelatCoron} and some efficient algorithms to better handle the solvability of kernels have been proposed  e.g. in \cite{JADACHOWSKI2012798}.

More recent contributions  focus on  coupled  parabolic PDEs \cite{Orlov_output_coupledparabolicSIAM2017},   with space varying reaction coefficients \cite{Vaszques2017parabolic75,Deutcher2017Solvekernels} and  time-  and space- varying coefficients \cite{Kerschbaum_Deutscher2019} all  of which  able to handle more challenging issues related to the solvability,   suitable choices of target systems,  coupling matrices structure and well-posedness issues in general.\\

In this paper, we aim at stabilizing scalar  1D  reaction–diffusion PDEs with a time- and space- varying reaction coefficient from a  different perspective. Our approach  combines  some ideas from hybrid systems, specifically from the framework of state-dependent switching laws, sampled-data and event-triggered sampling/control strategies.  For an overview of the literature on sampled-data,  event-triggered and switching strategies, we refer to e.g.   \cite{Hetel2017309,Liberzon2003,tabuada2007event,event-triger-Heeme-Johan-Tabu,girard2014dynamic,Postoyan_Aframework_ETS2014,JiangSmallGainETC,Liu_ZPJiang2015}  for finite-dimensional systems  and to  \cite{Logemann2005,Fridman2012826,KARAFYLLIS2018226,Selivanov_FridmanAuto,hante2009modeling},\cite{Tan2009,Prieur2014swithinghyperbolic,POLamare2015,Espitia2016_Aut,Espitia2018TAC,Polyakov2017_heatequ}  for some classes of infinite dimensional systems.

Having said that, the main ideas in this paper state that instead of handling a time-varying kernel capturing the time- and space- varying coefficient, we use a simpler kernel capturing only  the spatial variation of the reaction coefficient. This is possible  on the basis that  the reaction coefficient is sampled in time, thus the kernel-PDE reduces to a form involving space-varying coefficient  only  between two successive sampling times. In order to determine the time instants,  we introduce  event-triggered mechanisms that form an increasing sequence of triggering times (or time of events). At those event times,   kernels are computed/updated in aperiodic fashion and only when needed. In other words, kernels for the control are scheduled according to some event triggering condition (state-dependent law).  Doing so, the approach constitutes a kind of a gain scheduling strategy  suggesting then the adopted name for our approach: \textit{event-triggered gain scheduling.}
   
Sampling in time the reaction coefficient    introduces an error  (called error when sampling) that is reflected in the target system after transformation. This requires  the study of well-posedness issues, ISS properties (\cite{Karafyllis2019_book}) and the exponential stability of the closed loop system when the control gains are scheduled according to the event-triggered mechanisms. In this paper we propose two strategies: the first one  relies on a static triggering condition which takes into account the effect of the error when sampling after transformation and the current state of the closed-loop system. The second strategy relies on a dynamic triggering condition which makes uses of a dynamic variable that can be seen as a filtered version of the static triggering condition. Moreover, under the two proposed strategies, the avoidance of the so called  Zeno phenomenon is proved.  Hence, we can guarantee the well-posedness   as well as the  exponential stability of the closed-loop system.

 The paper is organized as follows. In Section \ref{section_problem_form}, we introduce the class of reaction-diffusion parabolic systems, the control design which includes the introduction of the event-triggered strategies for gain scheduling and the notion of existence and uniqueness of solutions. Section \ref{section:main-results}  provides  the main results which include the avoidance of the Zeno phenomenon, the well-posedness of the closed-loop system and the exponential stability result.   Section~\ref{numerical_simulation} provides a numerical example to illustrate the main results.  Finally, conclusions and perspectives are given in Section~\ref{conslusion_and_perspect}. The Appendix contains the proof of an auxiliary result.

\paragraph*{Notation}
$ \mathbb{R}_{+}$ will denote the set of nonnegative real numbers.   Let $S \subseteq \mathbb{R}^n$ be an open set and let $A \subseteq  \mathbb{R}^n$ be a set that satisfies $S \subseteq  A \subseteq  \bar{S}$. By $C^0(A;\Omega)$, we denote the class of continuous functions on $A$, which take values in $\Omega \subseteq  \mathbb{R}$. By $C^k(A;\Omega)$, where $k\geq 1$ is an integer, we denote the class of functions on $A$, which takes values in $\Omega$ and has continuous derivatives of order $k$. In other words, the functions of class $C^k(A;\Omega)$ are the functions which have continuous derivatives of order $k$ in $S=int(A)$ that can be continued continuously to all points in $\partial S \cap A$.  $L^2(0,1)$ denotes the equivalence class of Lebesgue measurable functions $f,g: [0,1] \rightarrow \mathbb{R}$ for which  $\Vert f\Vert =\left(\int_{0}^{1}  \vert f(x)\vert ^{2}dx\right)^{1/2} < \infty$ and  with inner product 	$\left<f,g\right> = \int_{0}^{1}   f(x)g(x)dx$.    Let $u:\mathbb{R}_{+}\times [0,1] \rightarrow \mathbb{R}$ be given. $u[t]$ denotes the profile of $u$ at certain $t\geq 0$, i.e. $(u[t])(x) = u(t,x)$, for all $x \in [0,1]$. For an interval $I \subseteq \mathbb{R}_{+}$, the space $C^{0}(I;L^2(0,1))$ is the space of continuous mappings $I\ni t  \mapsto  u[t] \in L^2(0,1)$. $H^2(0,1)$  denotes the Sobolev space of functions $f \in L^2(0,1)$ with square integrable (weak) first and second-order derivatives $f^{'}(\cdot), f^{''}(\cdot) \in L^2(0,1)$. 

\section{Problem description and control design}\label{section_problem_form}

Consider the following  scalar reaction-diffusion system with  time- and space- varying  reaction coefficient:
\begin{eqnarray}\label{eq:sysparabolic0}
u_t(t,x) & =&  \varepsilon u_{xx}(t,x) + \lambda(t,x) u(t,x) \\
u_x(t,0)&=&qu(t,0)\\
u(t,1)&=& U(t), \quad \text{or} \quad u_x(t,1)=U(t)   \label{BC_parabolic_PDE_u0}
\end{eqnarray} 
and initial condition:
\begin{equation}\label{IC_parabolic_PDE_u0}
 u(0,x)=u_{0}(x)
\end{equation}
where $\varepsilon >0$, $ q \in (-\infty, + \infty]$ (the case  $q=+\infty$ is interpreted as the Dirichlet case) and   $\lambda \in  \mathcal{C}^{0}(\mathbb{R}_{+}\times [0,1])$ with $\lambda[t] \in \mathcal{C}^{1}([0,1])$. $u: [0,\infty)\times[0,1]  \rightarrow \mathbb{R} $ is the system state  and  
   $U(t) \in \mathbb{R}$ is the control input. 
%
We assume that  $\lambda \in \mathcal{C}^{0}(\mathbb{R}_{+}\times [0,1])$ is bounded, i.e. there exists a constant $\bar{\lambda} > 0$ such that
\begin{equation}\label{eq:bound_lambda-tx}
\vert \lambda(t,x) \vert \leq \bar{\lambda}, \quad \forall t\geq 0, \quad x \in [0,1]
\end{equation}
Moreover, we assume the   following:
\begin{assumption}\label{Assumption_H}
 There exists a constant $\varphi > 0$  such that the following inequality holds:
\begin{equation}\label{Hypotesis_slowvarying}
 \vert \lambda(t,x) - \lambda(s,x) \vert \leq \varphi \vert t-s \vert, \quad \forall x \in [0,1], \quad t,s \geq 0
\end{equation}
\end{assumption}
Assumption \ref{Assumption_H} means that the reaction coefficient is Lipschitz with respect to time with Lipschitz constant $\varphi$. The constant $\varphi$ is a quantity that depends on the rate of change of the reaction coefficient: a high rate of change of the reaction coefficient implies a large value for $\varphi$. All subsequent results are also valid (with some modifications) if Assumption \ref{Assumption_H} is replaced by the less demanding assumption of Hölder continuity instead of Lipschitz continuity, i.e., if we replace the right hand side of \eqref{Hypotesis_slowvarying} by $\varphi|t-s|^a$, where $a \in (0,1)$ is a constant. However, for simplicity we will restrict the presentation of the results to the Lipschitz case.\\

In what follows, we do not consider system \eqref{eq:sysparabolic0}-\eqref{BC_parabolic_PDE_u0} as a time-varying system but we consider system \eqref{eq:sysparabolic0}-\eqref{BC_parabolic_PDE_u0} as a time-invariant system with two inputs: the control input $U(t)$ and the time-varying distributed parameter $\lambda[t]$. This is very important for theoretical reasons: since system \eqref{eq:sysparabolic0}-\eqref{BC_parabolic_PDE_u0} is not a time-varying system we can always assume that the initial time is zero. Therefore, the proposed event-triggered control scheme may be seen as a feedforward control scheme that compensates the effect of the distributed disturbance input $\lambda[t]$.

\subsection{Backstepping control design}
Our aim is the global exponential stabilization of the system  \eqref{eq:sysparabolic0}-\eqref{BC_parabolic_PDE_u0} at zero using boundary control.
To that end, we  follow the backstepping approach which makes uses of an invertible Volterra transformation to map the system into a target system simpler to handle and with desired stability properties.   Since the reaction coefficient is both time and space varying,  typically, the kernels of the transformation have to be chosen  to depend on time. This brings an additional source of complexity since  the resulting  kernel PDE equation contains a time-derivative of the kernel and involves  the time and space varying coefficient. Overall, the problem is much harder to  solve but has been the object of extensive investigation since the seminal work \cite{Colton1977181}.  Numerical strategies such as the method of successive approximation have been widely  employed (see e.g. \cite{MEURER20091182} and \cite{VasquezTrelatCoron}). 
 
Our approach takes a  different direction. We avoid solving  a  kernel-PDE hyperbolic spatial operator and first order derivative with respect to time  capturing the reaction coefficient (dependent on both time and space). We use simpler kernels for the control under which we  are still able to stabilize exponentially the closed-loop system. This brings  some degree of robustness to the controller.
Inspired by event-triggered control strategies (in both finite and infinite-dimensional settings), the key idea of our approach is to schedule the kernel gain at a certain increasing sequence of times. More precisely  the  computation and updating of the kernel are on events and only when needed. 
 The time instants are determined by  event-triggered mechanisms that form an increasing sequence$\{t_j\}_{j\in \mathbb{N}}$ with $t_0=0$  which will be characterized  later on. 
 
Let $\{t_j\}_{j\in \mathbb{N}}$ be an increasing sequence of times with $t_0=0$ and define:
\begin{equation}\label{eq:space-varying_reaction_coeff}
b_j(x):=\lambda(t_j,x), \quad \text{for} \quad  x\in [0,1]
\end{equation}
which is the sampled version of the reaction coefficient $\lambda(t,x)$. We define also  the error when sampling:
\begin{equation}\label{error_when_sampling}
e_j(t,x):= \lambda(t,x) -b_j(x), \quad \text{for} \quad  t \in [t_j,t_{j+1}), \quad  x\in [0,1]
\end{equation}
 Therefore, we can rewrite  \eqref{eq:sysparabolic0}-\eqref{BC_parabolic_PDE_u0}, for $t\in [t_j,t_{j+1})$ as follows:
\begin{eqnarray}
u_t(t,x) & =&  \varepsilon u_{xx}(t,x) + b_j(x) u(t,x) +  e_j(t,x)  u(t,x)  \label{eq:sysparabolic1}   \\
u_x(t,0)&=&q u(t,0)\\
u(t,1)&=& U(t), \quad \text{or} \quad u_x(t,1)=U(t)   \label{BC_parabolic_PDE_u1}
\end{eqnarray}
The backstepping boundary control design is performed by transforming \eqref{eq:sysparabolic1}-\eqref{BC_parabolic_PDE_u1} into a target system which will reflect of the error when sampling $e_j(t,x)$ \eqref{error_when_sampling}.
 Therefore,  consider the following  invertible Volterra transformation, for $j\geq 0$, 
\begin{equation}\label{backstepping _direct_trasf_1}
\begin{split}
w_j(t,x) &=  u(t,x) - \int_{0}^{x} K_j(x,y)u(t,y) dy := (\mathcal{K}_j u[t])(x)
\end{split}
\end{equation}
whose inverse is given by
\begin{equation}\label{backstepping _inverse_trasf_1}
\begin{split}
u(t,x) &=  w_j(t,x) + \int_{0}^{x} L_j(x,y)w(t,y) dy  := (\mathcal{L}_j w_j[t])(x)
\end{split}
\end{equation}
with kernels $K_j$,  $L_j \in \mathcal{C}^{2}(\mathcal{T})$  evolving in   a triangular domain given by $\mathcal{T}= \{ (x,y): 0 \leq y < x \leq 1 \}$ and satisfying \cite[Theorem 2]{Smyshlyaev-Krstic2004}: 
\begin{equation}\label{kernelPDEntimesn-equidifussivityComponentform}
K_{j,xx}(x,y)  - K_{j,yy}(x,y)    =  \frac{(b_j(y) + c)}{\varepsilon}K_j(x,y)  
\end{equation}
\begin{equation}\label{kernelPDEntimesn-equidifussivityComponentformBC1}
K_{jy}(x,0)=qK_{j}(x,0)   \hskip 1.3 cm
\end{equation}
\begin{equation}\label{kernelPDEntimesn-equidifussivityComponentformBC2}
K_j(x,x) = - \frac{1}{2\varepsilon}\int_{0}^{x}(b_j(s)+c)ds \hskip 0.8 cm
\end{equation}
\begin{equation}\label{kernelPDEntimesn-equidifussivityComponentform_inverse}
L_{j,xx}(x,y)  - L_{j,yy}(x,y)    =  -\frac{(b_j(x) + c)}{\varepsilon}L_j(x,y)  
\end{equation}
\begin{equation}\label{kernelPDEntimesn-equidifussivityComponentformBC1_inverse}
L_{jy}(x,0)= q L_{j}(x,0)\hskip 1.3 cm
\end{equation}
\begin{equation}\label{kernelPDEntimesn-equidifussivityComponentformBC2_inverse}
L_j(x,x) = - \frac{1}{2\varepsilon}\int_{0}^{x}(b_j(s)+c)ds \hskip 0.8 cm
\end{equation}
Under \eqref{backstepping _direct_trasf_1}, \eqref{kernelPDEntimesn-equidifussivityComponentform}-\eqref{kernelPDEntimesn-equidifussivityComponentformBC2} and selecting  the control $U(t)$ to satisfy
\begin{equation}\label{eq:control_feedback}
U(t) = \int_{0}^{1}K_{j}(1,y)u(t,y)dy, \quad t \in (t_j, t_{j+1}) 
\end{equation}
for Dirichlet actuation  or by
\begin{equation}\label{eq:control_feedback_Neumann actuation}
U(t) =  K_j(1,1)u(t,1) +  \int_{0}^{1}K_{j,x}(1,y)u(t,y)dy, \quad t \in (t_j, t_{j+1})
\end{equation}
for  Neumann actuation,   the  transformed system, for all $ j\geq 0 $,  $t \in (t_j,t_{j+1})$, is as follows:
\begin{eqnarray}
w_{j,t}(t,x) & = & \varepsilon w_{j,xx}(t,x) - c w_{j}(t,x) +  (\mathcal{K}_j f_j[t]) (x) \label{eq:target_system}  \\
w_{jx}(t,0)&=& qw_j(t,0)\\
w_{j}(t,1)&=&0, \quad \text{or}  \quad w_{jx}(t,1)=0 \  \label{eq:target_BC} 
\end{eqnarray} 
where 
\begin{equation}\label{functionalF1}
f_j(t,x) := e_j(t,x)u(t,x)
\end{equation}
and $c$ is a design parameter which is chosen as $c \geq \varepsilon \bar{q}^2$ (for Dirichlet actuation) or $c \geq \varepsilon \bar{q}^2 + \varepsilon/2$ (for Neumann actuation) where $\bar{q}=\max\{0,-q \}$ (see \cite{Smyshlyaev-Krstic2004}). 

Moreover, the following estimate holds, for all $j \geq 0$
\begin{equation}\label{eq:estimates_kernels}
\max\left\lbrace \vert K_j(x,y) \vert, \vert L_j(x,y) \vert \right\rbrace \leq M \exp(2M x), \quad \text{for} \quad (x,y)\in \mathcal{T}
\end{equation}
where $M:=\frac{\bar{\lambda} + c}{\varepsilon}$.

Definitions  \eqref{backstepping _direct_trasf_1},\eqref{backstepping _inverse_trasf_1}  imply  the following estimates, for all $j\geq 0$:
\begin{equation}\label{eq:estimate_Kj}
\Vert (\mathcal{K}_j u[t]) \Vert \leq \tilde{K}_j \Vert u[t] \Vert
\end{equation}
\begin{equation}\label{eq:estimate_Lj}
\Vert (\mathcal{L}_j w_j[t]) \Vert \leq \tilde{L}_j \Vert w[t] \Vert
\end{equation}
where $\tilde{K}_j$ and $\tilde{L}_j$ are defined, respectively by $\tilde{K}_j:= 1+ \left(\int_{0}^{1} \left( \int_{0}^{x}\vert K_j(x,y) \vert^2 dy \right) dx \right)^{1/2} $ and $ \tilde{L}_j:= \;  1+ \left(\int_{0}^{1} \left( \int_{0}^{x}\vert L_j(x,y) \vert^2 dy \right) dx \right)^{1/2}$. \newline In addition, inequality \eqref{eq:estimates_kernels} implies  the existence of a constant $G>0$ such that:
\begin{equation}\label{estimates_transformations_G}
\max \left\lbrace  \tilde{K}_j,  \tilde{L}_j \right\rbrace \leq G, \quad \text{for} \quad j \geq 0
\end{equation}
where
\begin{equation}\label{eq:G}
G:=1+\sqrt{\frac{\bar{\lambda}+c}{4\varepsilon}\left(\exp\left(\frac{4(\bar{\lambda}+c)}{\varepsilon}\right) -1 \right)}
\end{equation}
which is independent of $j$.

\subsection{Well-posedness analysis}
In order to study  well-posedness issues for \newline \eqref{eq:target_system}-\eqref{eq:target_BC}  and in turn for  \eqref{eq:sysparabolic1}-\eqref{BC_parabolic_PDE_u1} (by virtue of the bounded invertibility the backstepping transformations \eqref{backstepping _direct_trasf_1}-\eqref{backstepping _inverse_trasf_1}) under any  event-triggered  gain scheduling implementation,  we need first  a more general result (see Theorem \ref{Theore:notion_of_solution_general_setting} below)  which  establishes the   notion of solution for the following reaction-diffusion  system:  
\begin{equation}\label{eq:parabolic_pDE_general_setting_notion_solution}
\begin{aligned}
& w_t(t, x)=\varepsilon  w_{xx}(t, x)-c w(t, x)+(\mathcal{F}(t) w[t])(x),  \quad x \in(0,1)\\
&a_{0} w(t, 0)+b_{0} w_x(t, 0)=a_{1} w(t, 1)+b_{1} w_ x(t, 1)=0
\end{aligned}
\end{equation}
where $\varepsilon>0, c, a_{0}, b_{0}, a_{1}, b_{1}$ are constants with $a_{0}^{2}+b_{0}^{2}>0, a_{1}^{2}+b_{1}^{2}>0$ and for each $t \geq 0$
the operator $\mathcal{F}(t): L^{2}(0,1) \rightarrow L^{2}(0,1)$ is a linear bounded operator for which there exist constants $\Omega_1, \Omega_2>0$ such that
\begin{equation}\label{eq:constant_bound_1_operatorF}
 \Vert  \mathcal{F}[t]  \Vert \leq \Omega_1, \text { for all } t \geq 0
\end{equation}
\begin{equation}\label{eq:constant_bound_2_operatorF}
\Vert \mathcal{F}(t)-\mathcal{F}(s)\Vert \leq \Omega_2|t-s|, \text { for all } t, s \geq 0
\end{equation}
 \begin{theorem}\label{Theore:notion_of_solution_general_setting}
 For every initial condition  $w_{0} \in L^{2}(0,1)$ and $T>0,$ the initial-boundary value problem \eqref{eq:parabolic_pDE_general_setting_notion_solution} with
\begin{equation}\label{eq:initial_condition_general_seeting_notion_solution}
w[0]=w_{0}
\end{equation}
has a unique solution $w \in C^{0}\left([0,T];L^{2}(0,1)\right) \cap C^{1}\left((0, T) ; L^{2}(0,1)\right)$ with $w[t] \in D$ for all $t \in(0, T),$ where
\begin{equation*}
D:=\left\{f \in H^{2}(0,1): a_{0} f(0)+b_{0} f^{\prime}(0)=a_{1} f(1)+b_{1} f^{\prime}(1)=0\right\}
\end{equation*}
that satisfies \eqref{eq:initial_condition_general_seeting_notion_solution}  and \eqref{eq:parabolic_pDE_general_setting_notion_solution} for all $t\in (0,T)$.
\end{theorem}
\begin{proof}
See Appendix \ref{Proof_Theorem_well_posedness}.
\end{proof}

We are in position to   specialize the well-posedness result  to the system \eqref{eq:target_system}-\eqref{eq:target_BC} so as  we can construct the solution  by the step method. To do so, let us take in \eqref{eq:initial_condition_general_seeting_notion_solution}  $a_0=q$, $b_0=-1$ for  $q < + \infty$; $a_0=1$, $b_0=0$ for $q= + \infty $,   $a_1=1$, $b_1=0$ for Dirichlet actuation and    $a_1=0$,  $b_1=1$ for Neumann actuation. In addition, it suffices to observe  that the operator $(\mathcal{K}_{j} f_j[t])(x)$ in \eqref{eq:target_system} has the form of  $(\mathcal{F}(t)w[t])(x)$, i.e,
\begin{equation} \label{eq:operator_F_with_operatorTarget-system}
(\mathcal{F}(t)w[t])(x) = (\mathcal{K}_{j} (e_j[t] \mathcal{L}_{j}w_j[t]))(x), \quad t \in [t_j,t_{j+1})
\end{equation}
which is indeed the case   by virtue of  \eqref{error_when_sampling}, \eqref{backstepping _direct_trasf_1}, \eqref{backstepping _inverse_trasf_1} and \eqref{functionalF1}. Moreover, due to  \eqref{Hypotesis_slowvarying} in Assumption \ref{Assumption_H}, the operator $\mathcal{F}(t)$ satisfies \eqref{eq:constant_bound_1_operatorF}-\eqref{eq:constant_bound_2_operatorF}.   Extending continuously the operator $\mathcal{F}(t)$ defined by \eqref{eq:operator_F_with_operatorTarget-system} for $t\geq t_{j+1}$ and using Theorem \ref{Theore:notion_of_solution_general_setting} we obtain the following proposition:

\begin{proposition}\label{Proposition:existance_of_solution_target}
 For every initial data  $w_{j}[t_j]=(\mathcal{K}_{j}u[t_j])(x) \in L^{2}(0,1)$, there exists a unique function   $w_j \in C^{0}\left(\left[t_{j}, t_{j+1}\right] ; L^{2}(0,1)\right) \cap C^{1}\left((t_{j}, t_{j+1}); L^{2}(0,1)\right)$ with $w_j[t] \in D$ for $t \in (t_j,t_{j+1}]$  that satisfies \eqref{eq:target_system}-\eqref{eq:target_BC} for all $t\in (t_j,t_{j+1})$  where
$D \subset H^2([0,1])$ is the set of functions $f : [0,1] \rightarrow \mathbb{R} $ for which one has $f^{'}(0)=q f(0) $ and $f(1)=0$ for the case of Dirichlet  actuation  or  $f^{'}(1)=0$ for the case of Neumann actuation.
\end{proposition}
Consequently, by the bounded invertibility of the backstepping transformation, for every $u_0 \in L^2(0,1)$,  we can  construct a solution $u \in C ^{0} \left(\left[0, \lim_{j \rightarrow \infty} ( t_{j} )\right) ; L^{2}(0,1)\right) $ with  $u[t] \in H^2(0,1)$ for $t \in (0, \lim_{j \rightarrow \infty} (t_j))$  and $u  \in C^{1}(\tilde{I}; L^{2}(0,1))$ where $\tilde{I}=\left[0, \lim _{j \rightarrow \infty}\left(t_{j}\right)\right) \backslash\left\{t_{j}: j=0,1,2, \ldots\right\}$ which also satisfies \eqref{eq:sysparabolic1}-\eqref{BC_parabolic_PDE_u1}  for all $t\in \tilde{I}$.

\subsection{Event-triggered gain scheduling}
\vskip 0.5cm 
Let us  consider the following  Sturm-Liouville operator $B: D \rightarrow L^2(0,1)$ defined by
\begin{equation}\label{eq:SL_A}
(B h)(x)= - \varepsilon \frac{d^2 h}{dx^2}(x)
\end{equation}
for all $h \in D$ and $x \in (0,1)$ where $D \subset H^2([0,1])$ is the set of functions $h : [0,1] \rightarrow \mathbb{R} $ for which one has $h^{'}(0)=q h(0) $ and $h(1)=0$ for the case of Dirichlet  actuation  or  $h^{'}(1)=0$ for the case of Neumann actuation.

We denote  $\mu_1 < \mu_2 < ... < \mu_n< ..$ with $\lim_{n \rightarrow \infty} (\mu_n) = + \infty$     and  $\phi_n(x) \in \mathcal{C}^{2}([0,1],\mathbb{R})$, ($n=1,2...$) the eigenvalues and  the eigenfunctions, respectively,  of the operator $B$. 
\subsubsection{Event-triggered gain scheduling with a static triggering condition}
\vskip 0.5cm 
 We introduce the first event-triggering  strategy (or mechanism)  for gain scheduling considered in this paper. The triggering condition is a state-dependent law and   determines the time instants at which the reaction coefficient has to be sampled and thereby when the kernel computation/updating  has to be done.

\begin{definition}[Definition of the static event-triggered mechanism for gain scheduling]\label{Definition_event_based_sheduler}
Let   $K_j$ be the kernel  satisfying  \eqref{kernelPDEntimesn-equidifussivityComponentform}-\eqref{kernelPDEntimesn-equidifussivityComponentformBC2} and let $f_j(t,x)$ be given by \eqref{functionalF1},  $j\geq 0$. Let $\mu_1$ be the principal eigenvalue of the Sturm-Liouville operator $B$ \eqref{eq:SL_A}. Let $R \in (0,1)$ be a design parameter and define
\begin{equation}\label{eq:parameter_mu}
\mu:= c + \varepsilon \mu_1
\end{equation}
The  static event-triggered gain scheduler is defined as follows:

The times of events $t_j\geq 0$  with $t_0=0$  form a finite or countable set of times  which is determined by the following rules for some $j\geq0$: \\

\begin{itemize}
\item [a)] if $ \{   t >  t_{j} :  \Big< (\mathcal{K}_j u[t]), (\mathcal{K}_j f_j[t]) \Big>       >  \mu R  \Vert (\mathcal{K}_j u[t]) \Vert^2      \} = \emptyset$  then the set of the times of the events is $\{t_{0},...,t_{j}\}$.\\ 
\item [b)] if $\{   t >  t_{j} :    \Big< (\mathcal{K}_j u[t]), (\mathcal{K}_j f_j[t]) \Big>          >  \mu R   \Vert (\mathcal{K}_j u[t] )\Vert^2       \} \neq \emptyset$, then the next event time is given by:
 \begin{equation}\label{triggering_conditionISS_with_backstepping_original}
\begin{split} 
t_{j+1} :=   &  \inf \{   t >  t_{j} :   \Big< (\mathcal{K}_j u[t]), (\mathcal{K}_j f_j[t]) \Big>      > \mu R   \Vert (\mathcal{K}_j u[t]) \Vert^2      \}  
\end{split} 
\end{equation}  
where $u[t]$ denotes the solution of \eqref{eq:sysparabolic0}-\eqref{BC_parabolic_PDE_u0} with \eqref{eq:control_feedback} or \eqref{eq:control_feedback_Neumann actuation} for $t \geq t_j$. 
\end{itemize}
\end{definition}

\subsubsection{Event-triggered gain scheduling with a dynamic triggering condition}
\vskip 0.5cm 

 Inspired by \cite{girard2014dynamic} and \cite{Espitia2018TAC}, we introduce the second event-triggering mechanism for gain scheduling in this paper. It involves a dynamic variable which can be viewed as a filtered value of the static triggering condition in \eqref{triggering_conditionISS_with_backstepping_original}.  With this strategy we expect to reduce updating times for the kernel scheduling and obtain larger inter-execution times.

\begin{definition}[Definition of the  dynamic event-triggered mechanism for  gain scheduling]\label{Definition_event_based_sheduler_Dynamic}
Let   $K_j$ be the kernel  satisfying  \eqref{kernelPDEntimesn-equidifussivityComponentform}-\eqref{kernelPDEntimesn-equidifussivityComponentformBC2},  $f_j(t,x)$ be given by \eqref{functionalF1},  $j\geq 0$ and  $\mu$ be given by \eqref{eq:parameter_mu}. Let $R \in (0,1)$, $ \eta  \geq  2\mu (1-R)$ and $\theta>0$   be design parameters. 
\\
The  dynamic event-triggered gain scheduler is defined as follows:

The times of events $t_j\geq 0$  with $t_0=0$  form a finite or countable set of times  which is determined by the following rules for some $j\geq0$: \\

\begin{itemize}
\item [a)] if $ \{   t >  t_{j}  :     \Big< (\mathcal{K}_j u[t]), (\mathcal{K}_j f_j[t]) \Big>        - \mu R   \Vert (\mathcal{K}_j u[t]) \Vert^2 >  \frac{1}{\theta}m(t)      \} = \emptyset$  then the set of the times of the events is $\{t_{0},...,t_{j}\}$.\\ 
\item [b)] if $\{   t >  t_{j}  :    \Big< (\mathcal{K}_j u[t]), (\mathcal{K}_j f_j[t]) \Big>          - \mu R   \Vert (\mathcal{K}_j u[t]) \Vert^2 >  \frac{1}{\theta}m(t)      \} \neq \emptyset$, then the next event time is given by:
 \begin{equation}\label{triggering_conditionISS_with_backstepping_original_Dynamic}
\begin{split} 
t_{j+1} :=   &  \inf \{   t >  t_{j}  :  \Big< (\mathcal{K}_j u[t]), (\mathcal{K}_j f_j[t]) \Big>    - \mu R   \Vert (\mathcal{K}_j u[t]) \Vert^2    > \frac{1}{\theta}m(t)      \}  
\end{split} 
\end{equation}  
where $u[t]$ denotes the solution of \eqref{eq:sysparabolic0}-\eqref{BC_parabolic_PDE_u0} with \eqref{eq:control_feedback} or \eqref{eq:control_feedback_Neumann actuation} for $t >t_j$ and 
 $m$ satisfies the ordinary differential equation
\begin{equation}\label{eq:internal_dynamic_variable}
\dot{m}(t) = - \eta m(t)  + \left( \mu R \Vert (\mathcal{K}_j u[t]) \Vert^2  -    \Big< (\mathcal{K}_j u[t]), (\mathcal{K}_j f_j[t]) \Big>    \right), \quad  \text{for}\quad   t \geq t_j
\end{equation}
 and we set $m(t_j) =0$. 
\end{itemize}
\end{definition}
\begin{remark}\label{remark:paramteres_two_event_strategies}
Let us remark  that the static event-triggered strategy  has only one design parameter (i.e. $R \in (0,1)$) whereas the dynamic event-triggered strategy has three additional design parameters, namely, $R$ (as in the static case),  $\eta$ and   $\theta$.  Essentially, $\eta$ adjusts the convergence rate of the filter \eqref{eq:internal_dynamic_variable}  that  can be  characterized as  $\eta \geq 2 \mu (1-R)$. The parameter $\theta$, on the other hand,   can be selected to contribute to sample less frequent than with the static event-triggered strategy. As a matter of  fact,  one can see the static event-triggering condition  \eqref{triggering_conditionISS_with_backstepping_original} as the limiting case  of the dynamic event-triggering condition \eqref{triggering_conditionISS_with_backstepping_original_Dynamic}-\eqref{eq:internal_dynamic_variable}  when $\theta$ goes to $ + \infty$. 
\end{remark}
The following result guarantees that the dynamic variable $m(t)$ remains always  positive between two successive triggering times. This fact is going to be helpful in the stability analysis of the closed-loop system.
\begin{lemma}\label{Lemma_dynamic_variable}
Under the definition of the event-triggered
gain scheduling    with dynamic trigger condition \eqref{triggering_conditionISS_with_backstepping_original_Dynamic}-\eqref{eq:internal_dynamic_variable}, it holds, for $t\in [t_{j},t_{j+1})$, $j\geq 0$, that $ \frac{1}{\theta}m(t) + \mu R   \Vert (\mathcal{K}_j u[t]) \Vert^2 -  \Big< (\mathcal{K}_j u[t]), (\mathcal{K}_j f_j[t]) \Big>   \geq 0  $ and $m(t)\geq 0$.
\end{lemma}
\begin{proof}
From definition of the  the event-triggered gain scheduling     with dynamic triggering condition \eqref{triggering_conditionISS_with_backstepping_original_Dynamic}-\eqref{eq:internal_dynamic_variable}, events are triggered to guarantee, for $t\in[t_j,t_{j+1})$,  $j \geq 0$ that $ \frac{1}{\theta}m(t) + \mu R   \Vert (\mathcal{K}_j u[t]) \Vert^2 -  \Big< (\mathcal{K}_j u[t]), (\mathcal{K}_j f_j[t]) \Big>   \geq 0 $. This inequality in conjunction with \eqref{eq:internal_dynamic_variable} yields:
\begin{equation}
\dot{m}(t) \geq -(\eta+\frac{1}{\theta}) m(t)  
\end{equation}
for which the Comparison principle can be used to guarantee $m(t) \geq 0$, for all  $t\in[t_j,t_{j+1})$,  $j \geq 0$  and provided that $m(t_j)\geq 0$.  
\end{proof}
Lemma \ref{Lemma_dynamic_variable} guarantees that $\tfrac{1}{\theta}m(t_{j+1}) + \mu R   \Vert (\mathcal{K}_j u[t_{j+1}]) \Vert^2 -  \big< (\mathcal{K}_j u[t_{j+1}]), (\mathcal{K}_j f_j[t_{j+1}]) \big> \newline  \geq 0$ and that $m(t_{j+1})\geq 0$   when $t_{j+1}<+ \infty$ (by continuity).

\begin{proposition}\label{proposition:first_triggering_time}
If the time of the next event generated by  \eqref{triggering_conditionISS_with_backstepping_original}  is finite, then the time of the next event generated by  the dynamic event triggered mechanism \eqref{triggering_conditionISS_with_backstepping_original_Dynamic}-\eqref{eq:internal_dynamic_variable}  is strictly larger than the time of the next event generated by the  static event triggered mechanism \eqref{triggering_conditionISS_with_backstepping_original}.

\end{proposition}
\begin{proof}
Without loss of generality we may assume that $j=0$ (and consequently $t_0=0$). Notice that if $u[0]=0$ then both the static and dynamic event triggering conditions give  $t_1=+\infty$. By assumption, the time of the next event generated by the static strategy is finite; therefore it holds that $u[0]$ is not zero. Consequently, $\mathcal{K}_0 u[0]$ is not zero. \\
Let $t_1$ be the time of the next event generated by  the static event triggered mechanism and let $T$ be the time of the next event generated by  the dynamic one. We  show next that $T>t_1$ by contradiction. Assume that $T\leq t_1$. Define
\begin{equation}\label{eq:q_first_triggering_time}
q(t):= \mu R   \Vert (\mathcal{K}_0 u[t]) \Vert^2 -  \big< (\mathcal{K}_0 u[t]), (\mathcal{K}_0 f_0[t]) \big>
\end{equation}
Then we have by virtue of \eqref{triggering_conditionISS_with_backstepping_original}, \eqref{eq:q_first_triggering_time}  $q(t) \geq 0$ for all $t \in [0,t_1]$ and  by virtue of \eqref{triggering_conditionISS_with_backstepping_original_Dynamic}, \eqref{eq:q_first_triggering_time}  $q(T)=-\frac{1}{\theta}m(T)$, implying  that $m(T)\leq 0$. Since $m(0)=0$ and $\dot{m}(t) = -\eta m(t) +q(t)$ for all $t  \in [0,T]$, we have
\begin{equation}
m(t)=  \int_{0}^{t} \exp(-\eta (t-s))q(s)ds, \quad \text{for all} \quad t \in [0,T]
\end{equation} 
Since $q(t)\geq 0$ for all $t \in [0,T]$ we get $m(T)\geq 0$ and thus we conclude that $m(T)=0$.  By continuity of $q(t)$ and the fact that $q(t) \geq 0$ for all $t \in [0,T]$, the integral $\int_{0}^{T}\exp(-\eta(T-s))q(s)ds $ is  zero  only if $q(t)$ is identically zero on $[0,T]$. However,  that is not possible since $f_0[0]=0$ (recall \eqref{eq:space-varying_reaction_coeff},\eqref{error_when_sampling} and \eqref{functionalF1})   and since $q(0)=  \mu R \Vert \mathcal{K}_0 u[0]\Vert^2  - \big <\mathcal{K}_0 u[0],K_0f_0[0]\big>  =  \mu R \Vert \mathcal{K}_0 u[0]\Vert^2  >0 $. Thus, it must hold that $T> t_1$. 
\end{proof}

\section{Analysis of the closed-loop system and main results}\label{section:main-results}

In this section we present our main results: the avoidance for the Zeno behavior \footnote{We recall,  the Zeno phenomenon  means infinite triggering times in a finite-time interval.  In practice, Zeno phenomenon would represent infeasible implementation into digital platforms since one would require to sample infinitely fast.},  the well-posedness and the exponential
stability of the closed-loop system under boundary controller whose gains are scheduled according to the two event-triggered  strategies. 

\subsection{Avoidance of the Zeno phenomenon}

\subsubsection{Event-triggered gain scheduling with a static triggering condition}\label{Subsection_dwell_time_static}

\begin{proposition}\label{theo:minimal_dwellt_time}
Under  \eqref{triggering_conditionISS_with_backstepping_original},  there exists a minimal dwell-time between two triggering times, i.e.  there exists a constant $\tau >0$ (independent of the initial condition $u_0$) such that $t_{j+1} - t_{j} \geq \tau $,  for all $j \geq 0$.  More specifically,  $\tau$ satisfies: 
\begin{equation}\label{eq:minimal_dwell_time01}
\tau=\frac{1}{\varphi}\frac{\mu R}{G^2}
\end{equation}
with $\mu= c + \varepsilon \mu_1$ (recall \eqref{eq:parameter_mu} with $\mu_1$ being the principal eigenvalue of the Sturm-Liouville operator $B$  \eqref{eq:SL_A}), $R \in (0,1)$ being the  design parameter involved in the event-triggering condition \eqref{triggering_conditionISS_with_backstepping_original}, $\varphi$ as in Assumption \ref{Assumption_H}  and $G$ given by \eqref{eq:G}.
\end{proposition}
\begin{proof}
 Assume that an event occurs at   $t = t_{j+1}$, Then,   from  and using  \eqref{backstepping _direct_trasf_1}, continuity of all mappings involved with respect to time and the Cauchy-Schwarz inequality, the following  more conservative estimate holds:
\begin{equation}\label{triggering_cond_at_tn+1}
\Vert  w[t_{j+1}]) \Vert  \Vert  (\mathcal{K}_j f_j[t_{j+1}]) \Vert        \geq  \mu R  \Vert w[t_{j+1}] \Vert^2 
\end{equation} 
Using  \eqref{backstepping _inverse_trasf_1}, \eqref{functionalF1}, \eqref{eq:estimate_Kj}-\eqref{eq:G} we get from \eqref{triggering_cond_at_tn+1}:
\begin{equation}
G^2 \Vert w[t_{j+1}] \Vert^2 \Vert e_j[t_{j+1}] \Vert_{\infty}  \geq \mu R  \Vert w[t_{j+1}] \Vert^2
\end{equation}
Therefore,
\begin{equation}
G^2 \Vert e_j[t_{j+1}] \Vert_{\infty}  \geq \mu R 
\end{equation}
By virtue of Assumption \ref{Assumption_H} in conjunction with \eqref{eq:space-varying_reaction_coeff}  and \eqref{error_when_sampling}, we obtain for all $j \geq 0$
\begin{equation}
G^2 \varphi (t_{j+1}-t_j)  \geq \mu R 
\end{equation}
from which we can deduce (using definition \eqref{eq:minimal_dwell_time01})
\begin{equation}\label{eq:minimal_dwell_time0}
  t_{j+1}-t_j \geq \tau
\end{equation}
being $\tau$ the minimal dwell-time (independent on initial conditions).
 \end{proof}
Proposition \ref{theo:minimal_dwellt_time} allows us to conclude that $\lim(t_j)=+\infty$ and thereby we can apply Proposition \ref{Proposition:existance_of_solution_target} to  get the following result  on  the existence of solutions of the closed-loop system  \eqref{eq:sysparabolic0}-\eqref{IC_parabolic_PDE_u0}  with control \eqref{eq:control_feedback}, under the static event-triggered gain scheduler    \eqref{triggering_conditionISS_with_backstepping_original}.
\begin{corollary}\label{existence_and_continuity_solutions_on_R}
For every initial condition $u_0 \in L^2(0,1)$, there exists a unique mapping $u \in  C^{0}(\mathbb{R}_{+}; L^{2}(0,1)) \cap C^{1}(I; L^{2}(0,1))$, $u[t] \in H^2(0,1)$ for $t>0$ that satisfies \eqref{eq:sysparabolic0}-\eqref{IC_parabolic_PDE_u0}  with control \eqref{eq:control_feedback}  or \eqref{eq:control_feedback_Neumann actuation}, under the static event-triggered gain scheduler   \eqref{triggering_conditionISS_with_backstepping_original} for all $t \in I $ and $I = \mathbb{R}_{+} \backslash \{ t_{j} \geq 0, j= 0,1,2,...\} $.
\end{corollary}
\begin{proof}
It is an immediate  consequence of Proposition \ref{Proposition:existance_of_solution_target} and Proposition \ref{theo:minimal_dwellt_time}.  Indeed, the solution is constructed  (by the step method) iteratively between successive triggering times.  
\end{proof} 

\begin{remark}
The minimal dwell-time depends on the rate of change of the reaction coefficient. A higher rate of change of the reaction coefficient (i.e., large $\varphi$) would give a smaller minimal dwell-time (i.e., more frequent event triggering). This is expected since a high rate of change of the reaction coefficient requires a more frequent update of the control law.
\end{remark}

\subsubsection{Event-triggered gain scheduling with a dynamic triggering condition}
 By virtue of Proposition \ref{proposition:first_triggering_time}, the inter-execution time for the dynamic event triggered mechanism \eqref{triggering_conditionISS_with_backstepping_original_Dynamic}-\eqref{eq:internal_dynamic_variable} always exceeds the inter-execution time for the static event triggered mechanism \eqref{triggering_conditionISS_with_backstepping_original}.  Due to Proposition \ref{theo:minimal_dwellt_time}, the Zeno
phenomenon is immediately excluded.

Therefore, as in the static case,  we can also conclude that $\lim(t_j)=+\infty$ and thereby we can apply Proposition \ref{Proposition:existance_of_solution_target} to  get the following result  on  the existence of solutions of the closed-loop system  \eqref{eq:sysparabolic0}-\eqref{IC_parabolic_PDE_u0}  with control \eqref{eq:control_feedback}  or \eqref{eq:control_feedback_Neumann actuation}, under the dynamic event-triggered gain scheduler    \eqref{triggering_conditionISS_with_backstepping_original_Dynamic}-\eqref{eq:internal_dynamic_variable}.

\begin{corollary}\label{existence_and_continuity_solutions_on_R_2}
For every initial condition $u_0 \in L^2(0,1)$, there exists a unique mapping $u \in  C^{0}(\mathbb{R}_{+}; L^{2}(0,1)) \cap C^{1}(I; L^{2}(0,1))$, $u[t] \in H^2(0,1)$ for $t>0$ that satisfies \eqref{eq:sysparabolic0}-\eqref{IC_parabolic_PDE_u0}  with control \eqref{eq:control_feedback} or \eqref{eq:control_feedback_Neumann actuation}, under the dynamic event-triggered gain scheduler \eqref{triggering_conditionISS_with_backstepping_original_Dynamic}-\eqref{eq:internal_dynamic_variable} for all $t \in I $ and $I = \mathbb{R}_{+} \backslash \{ t_{j} \geq 0, j= 0,1,2,...\} $.
\end{corollary}

\subsection{Exponential stability analysis}
\vskip 0.5cm 
We present next the stability results under our two event-triggered gain scheduling strategies.
\subsubsection{Event-triggered gain scheduling with a static triggering condition}

\vskip 0.4cm 
 \begin{theorem}\label{main_theorem_EBC_backstepping}
 Under Assumption \ref{Assumption_H},  if   the following condition is fulfilled,
\begin{equation}\label{eq:condition_gamma_tau_sigma}
\varphi  <\frac{\mu^2 R (1-R)}{G^2\ln\left( G  \right)},
\end{equation}
where $\varphi$, $G$, $\mu$, $R$ are defined, respectively, in \eqref{Hypotesis_slowvarying}, \eqref{eq:G}, \eqref{eq:parameter_mu}, \eqref{triggering_conditionISS_with_backstepping_original};
then, the  closed-loop system  \eqref{eq:sysparabolic0}-\eqref{IC_parabolic_PDE_u0}  with control \eqref{eq:control_feedback}  or \eqref{eq:control_feedback_Neumann actuation}, under the static event-triggered gain scheduler    \eqref{triggering_conditionISS_with_backstepping_original},  is globally exponentially stable. More specifically, there exists a constant $\sigma>0$  
such that:
\begin{equation}\label{eq:L2_GES}
\Vert u[t] \Vert \leq G  \exp(-\sigma t) \Vert u[0] \Vert, \quad \text{for all} \quad t\geq 0
\end{equation}
 \end{theorem}
 \begin{proof}
 By using the variational characterization of eigenvalues (see  \cite[Section 11.4]{Strauss2008})  in conjunction with \eqref{eq:target_system}-\eqref{eq:target_BC},\eqref{functionalF1},    the following estimate holds for all $t\in (t_j,t_{j+1})$:
 \begin{equation}\label{eq:Wirtinger_target0}
 \frac{d}{dt} \left( \frac{1}{2} \Vert w_j[t] \Vert^2 \right) \leq - \mu \Vert w_j[t] \Vert^2 +  \left\langle w_j[t],(\mathcal{K}_jf_j[t]) \right \rangle
 \end{equation}
 where $\mu= c + \varepsilon \mu_1$ (recall \eqref{eq:parameter_mu} $\mu_1$ being the principal eigenvalue of the Sturm-Liouville operator $B$  \eqref{eq:SL_A}). We can rewrite \eqref{eq:Wirtinger_target0} as follows:
 \begin{equation}\label{eq:Wirtinger_target2}
 \frac{d}{dt} \left( \frac{1}{2} \Vert w_j[t] \Vert^2 \right) \leq - \mu(1-R) \Vert w_j[t] \Vert^2  - \mu R \Vert w_j[t] \Vert^2  +   \Big< w_j[t], (\mathcal{K}_j f_j[t]) \Big>   
 \end{equation}
where $R \in (0,1)$ is the  parameter  involved in   \eqref{triggering_conditionISS_with_backstepping_original}.

  Therefore,  from the definition of the  static  event-triggered gain scheduler,  events are triggered to guarantee,  $  \Big< w_j[t], (\mathcal{K}_j f_j[t]) \Big>   \leq \mu R \Vert w[t] \Vert^2$, for all $t\in (t_j,t_{j+1})$. Then, we obtain for all $t\in [t_j,t_{j+1} )$:
  \begin{equation}\label{eq:Wirtinger_target3}
 \Vert w_j[t] \Vert^2  \leq \exp\left( - 2\mu(1-R) (t-t_j) \right)  \Vert w_j[t_j] \Vert^2  
 \end{equation}
Using \eqref{backstepping _direct_trasf_1}, \eqref{backstepping _inverse_trasf_1}, \eqref{eq:estimate_Kj}-\eqref{eq:G} and \eqref{eq:Wirtinger_target3}, we get:  
\begin{equation}\label{eq:estimate_original_system2}
\Vert u[t] \Vert^2 \leq  G^2 \exp(- 2\mu(1-R)(t- t_j))\Vert u[t_j] \Vert^2 
\end{equation}
for all $t \in [t_j,t_{j+1})$. 
Since $u \in \mathcal{C}^{0}(\mathbb{R}_{+}; L^2(0,1))$, it follows that \eqref{eq:estimate_original_system2} holds for $t=t_{j+1}$ as well, i.e.
\begin{equation}\label{eq:estimate_original_system3}
\Vert u[t_{j+1}] \Vert^2 \leq  G^2\exp(-2\mu(1-R)(t_{j+1}- t_j))\Vert u[t_j] \Vert^2 
\end{equation}
    Now, for all $t\geq 0$, an estimate of $\Vert u[t] \Vert$ in terms of $\Vert u[0] \Vert$ can be derived recursively, by using \eqref{eq:minimal_dwell_time0} and the   fact that there have been   $j $ events and that  $ j \tau$ units of time have (at least) been passed until $t$. 
 To that end,  we can apply induction  on $j$  and prove that, for all $j\geq 0$,
\begin{equation}\label{eq:eq:estimate_original_system3_towards_intialcondition}
\Vert u[t_{j}] \Vert^2 \leq  (G^{2j})\exp(- 2\mu(1-R)t_j)  \Vert u[0] \Vert^2 
\end{equation}
and that $ t_j \geq  j\tau$. Let $j\geq 0$ be given (arbitrary) and $t \in [t_j,t_{j+1})$ (arbitrary).  We obtain from  \eqref{eq:estimate_original_system2},\eqref{eq:estimate_original_system3} and \eqref{eq:eq:estimate_original_system3_towards_intialcondition}:
\begin{equation}
\Vert u[t] \Vert^2 \leq  (G^2)^{j+1} \exp(-2\mu(1-R) t)\Vert u[0] \Vert^2 
\end{equation}
Moreover, since  $j \leq \frac{t}{\tau} $, it holds:
\begin{equation}
\Vert u[t] \Vert^2 \leq  G^2 \exp\left(-\left(2\mu(1-R) - \tfrac{2ln(G)}{\tau}\right) t\right)\Vert u[0] \Vert^2 
\end{equation}
In light of condition \eqref{eq:condition_gamma_tau_sigma} in conjunction with \eqref{eq:minimal_dwell_time0} where  
$\tau = \frac{1}{\varphi}\frac{\mu R}{G^2}$, 
we finally obtain:
\begin{equation}
\Vert u[t] \Vert \leq G \exp(-\sigma t) \Vert u[0] \Vert, \quad \text{for all} \quad t\geq 0
\end{equation}
with $\sigma=\frac{\mu^2 R(1-R) - \varphi G^2 \ln(G)}{\mu R} >0 $.  
 This concludes the proof. 
 \end{proof}
 \begin{remark}
Notice that  $\mu^2 R(1-R) - \varphi G^2 \ln(G) >0 $  holds true  provided that $\varphi$ is sufficiently small (this corresponds to the case where $\lambda(t,x)$ is slowly time-varying coefficient).  
In addition, it is worth remarking that we can select $R=\frac{1}{2}$ in order to maximize the allowable upper bound $\varphi$. Nevertheless, different values of $R$ may be used in practice since the obtained estimates are conservative.
The proof of Theorem 2 provides a (conservative) explicit estimate of the convergence rate $\sigma>0$. The obtained estimate shows that the smaller $\varphi$ is (i.e., the slower the change of the reaction coefficient), the higher the convergence rate is.  
\end{remark}

\subsubsection{Event-triggered gain scheduling with a dynamic triggering condition} 

 \begin{theorem}\label{main_theorem_EBC_backstepping_Dynamic}
If condition \eqref{eq:condition_gamma_tau_sigma} of Theorem \ref{main_theorem_EBC_backstepping} holds, then,   the  closed-loop system  \eqref{eq:sysparabolic0}-\eqref{IC_parabolic_PDE_u0}  with control \eqref{eq:control_feedback}  or \eqref{eq:control_feedback_Neumann actuation},  under dynamic  event-triggered gain scheduler    \eqref{triggering_conditionISS_with_backstepping_original_Dynamic}-\eqref{eq:internal_dynamic_variable}  is globally exponentially stable. More specifically, there exists a constant $\sigma>0$  
such that:
\begin{equation}\label{eq:L2_GES_2}
\Vert u[t] \Vert \leq G \exp(-\sigma t) \Vert u[0] \Vert, \quad \text{for all} \quad t\geq 0
\end{equation}
 \end{theorem}
 \begin{proof}


An estimate of the time-derivative of the following function $W(t):=\frac{1}{2}\Vert w_j[t] \Vert^2 + m(t)$ along the solutions \eqref{eq:target_system}-\eqref{eq:target_BC},\eqref{functionalF1},\eqref{eq:internal_dynamic_variable} is given by:
\begin{equation}
\dot{W}(t) \leq  - \mu \Vert w_j[t] \Vert^2 +   \Big< w_j[t], (\mathcal{K}_j f_j[t]) \Big>    - \eta m(t) -  \Big< w_j[t], (\mathcal{K}_j f_j[t]) \Big>   +\mu R \Vert w_j[t] \Vert^2
\end{equation}
which can be rewritten as follows:
\begin{equation}
\dot{W}(t) \leq  - \mu(1-R)  \left( \Vert w_j[t] \Vert^2  + 2m(t) \right)  - m(t) (\eta
 - 2\mu (1-R))
\end{equation}
By Lemma \ref{Lemma_dynamic_variable},  we guarantee that $ m(t) \geq 0 $ and since  $\eta  \geq  2\mu(1-
R)$ (recall Definition \ref{Definition_event_based_sheduler_Dynamic}), thus we get:
\begin{equation}\label{eq:stric_decreasing_W}
\dot{W}(t) \leq  - 2 \mu(1-R) W(t)  
\end{equation}
Therefore,  we obtain for $t\in [t_j,t_{j+1})$:
\begin{equation}\label{eq:Wirtinger_target3_dynamic2}
\frac{1}{2}\Vert w_j[t] \Vert^2 +  m(t) \leq  \exp(- 2\mu(1-R)(t-t_j ))  \left( \frac{1}{2}\Vert w_j[t_j] \Vert^2  + m(t_j) \right)  
\end{equation}
Notice that $\frac{1}{2}\Vert w_j[t] \Vert^2  \leq \frac{1}{2}\Vert w_j[t] \Vert^2 +  m(t) $ and that by Definition \ref{Definition_event_based_sheduler_Dynamic}, $m(t_j) =0$. Therefore, from \eqref{eq:Wirtinger_target3_dynamic2} we have, for all $t\in [t_j,t_{j+1})$: 
\begin{equation}
\Vert w_j[t] \Vert^2  \leq  \exp(- 2\mu(1-R)(t-t_j)) \Vert w_j[t_j] \Vert^2    
\end{equation}
The remaining part of the proof follows the same reasoning as the proof of Theorem \ref{main_theorem_EBC_backstepping}   (see from \eqref{eq:estimate_original_system2}).  This concludes the proof.
\end{proof}
\begin{remark}
The function $W(t)$ is monotonically decreasing  (see \eqref{eq:stric_decreasing_W}), for all $t\in [t_j, t_{j+1})$. However, the function  $\frac{1}{2}\Vert w_j[t] \Vert^2$ may not be  monotonically decreasing on that interval. The design parameter $\theta$ involved in the dynamic event-triggering condition \eqref{triggering_conditionISS_with_backstepping_original_Dynamic}-\eqref{eq:internal_dynamic_variable} and also discussed in Remark \ref{remark:paramteres_two_event_strategies}, allows to limit the potential increase of $\frac{1}{2}\Vert w_j[t] \Vert^2$.  Indeed, since events are triggered to guarantee  $ \frac{1}{\theta}m(t)  + \mu R   \Vert (\mathcal{K}_j u[t]) \Vert^2 - \Big< (\mathcal{K}_j u[t]), (\mathcal{K}_j f_j[t]) \Big>  \geq 0$, it holds that:
\begin{equation*}
 \frac{d}{dt} \left( \frac{1}{2} \Vert w_j[t] \Vert^2 \right) \leq - \mu(1-R) \Vert w_j[t] \Vert^2 +  \frac{1}{\theta}m(t)
 \end{equation*}
Notice that the larger the value of $\theta$, the more limited the increase. We approach then to the case as we were  dealing with the static event-triggered gain scheduler.
\end{remark}
\section{Numerical simulations}\label{numerical_simulation}
We illustrate the results by considering  \eqref{eq:sysparabolic0}-\eqref{IC_parabolic_PDE_u0}  with $\varepsilon= 1$, $q=+\infty$  and initial condition   $u_0(x) = 2(x - x^2)$. For numerical simulations, the state of the system has been discretized by divided differences on a uniform grid with the step $h=0.02$ for the space variable. The discretization with respect to time was done using the implicit Euler scheme with step size $\Delta t=h^2 $. We run simulations on a frame of 2s.
 We choose the time- and space- varying coefficient $\lambda(t,x)$ to have  a simple form as $ \lambda(t,x)= \lambda_{c} + \lambda_{t}(t) + \lambda_{x}(x)$. More specifically:
\begin{equation}\label{lambda_ex1}
\lambda(t,x) = 10  + \frac{50}{\cosh^2(5(t-1))}+ 7\cos(5\pi t) + \frac{50}{\cosh^2(5x)}, \quad t>0,\quad  x \in  [0,1] 
\end{equation}
which has a profile depicted in Figure \ref{profile_reaction-term}.
\begin{figure*}[t]
\centering{
\subfigure{\includegraphics[width=0.52\columnwidth]{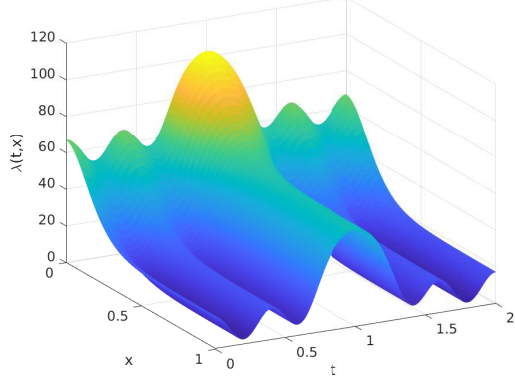} }
\caption{Profile of  the time- and space- varying reaction coefficient $\lambda(t,x)=10  + \frac{50}{\cosh^2(5(t-1))}+ 7\cos(5\pi t) + \frac{50}{\cosh^2(5x)}$ for the reaction-diffusion system \eqref{eq:sysparabolic0}-\eqref{IC_parabolic_PDE_u0}. }
\label{profile_reaction-term}
}
\end{figure*}
We  stabilize the closed-loop system \eqref{eq:sysparabolic0}-\eqref{IC_parabolic_PDE_u0}  under Dirichlet actuation with  boundary control \eqref{eq:control_feedback}    whose
kernel gains satisfy \eqref{kernelPDEntimesn-equidifussivityComponentform}-\eqref{kernelPDEntimesn-equidifussivityComponentformBC2}  and are scheduled according to the two event-triggered mechanisms we introduced in Definition \ref{Definition_event_based_sheduler}  (static-based triggering condition) and Definition \ref{Definition_event_based_sheduler_Dynamic} (dynamic-based triggering condition).  

The parameters of the triggering conditions are set   $R=0.15$, $\mu = c + \varepsilon \pi^2 = \pi^2$ with $c=0$. In addition,  $\eta =16.7$  and  $\theta=0.15$.\\   
From \eqref{eq:space-varying_reaction_coeff},  $b_j(x) = \lambda(t_j,x) = \tilde{\lambda}_j + \frac{50}{\cosh^2(5x)} $, where $\tilde{\lambda}_j:=10  + \frac{50}{\cosh^2(5(t_j-1))} + 7\cos(5 \pi t_j) $ with $\{t_j\}_{j\in \mathbb{N}}$ according to  \eqref{triggering_conditionISS_with_backstepping_original} (static)  or \eqref{triggering_conditionISS_with_backstepping_original_Dynamic}-\eqref{eq:internal_dynamic_variable} (dynamic). In either cases,   kernels $K_j$ satisfying \eqref{kernelPDEntimesn-equidifussivityComponentform}-\eqref{kernelPDEntimesn-equidifussivityComponentformBC2}, for all $t\in [t_j,t_{j+1})$, admit a closed-form solution which is given as follows \cite[Section VIII.E]{Smyshlyaev-Krstic2004}:
\begin{equation}\label{Kernel_example1}
K_j(x, y)=- \tilde{\lambda}_j y \frac{I_{1}\left(\sqrt{\tilde{\lambda}_j\left(x^{2}-y^{2}\right)}\right)}{\sqrt{\tilde{\lambda}_j\left(x^{2}-y^{2}\right)}}  -5 \tanh (5 y) I_{0}\left(\sqrt{\tilde{\lambda}_j\left(x^{2}-y^{2}\right)}\right)
\end{equation}
where $I_{m}(\cdot)$,  $m \in \mathbb{Z}$  is a modified Bessel function of the first kind of order $m$.
\begin{figure*}[t]
\centering{
\subfigure{\includegraphics[width=0.49\columnwidth]{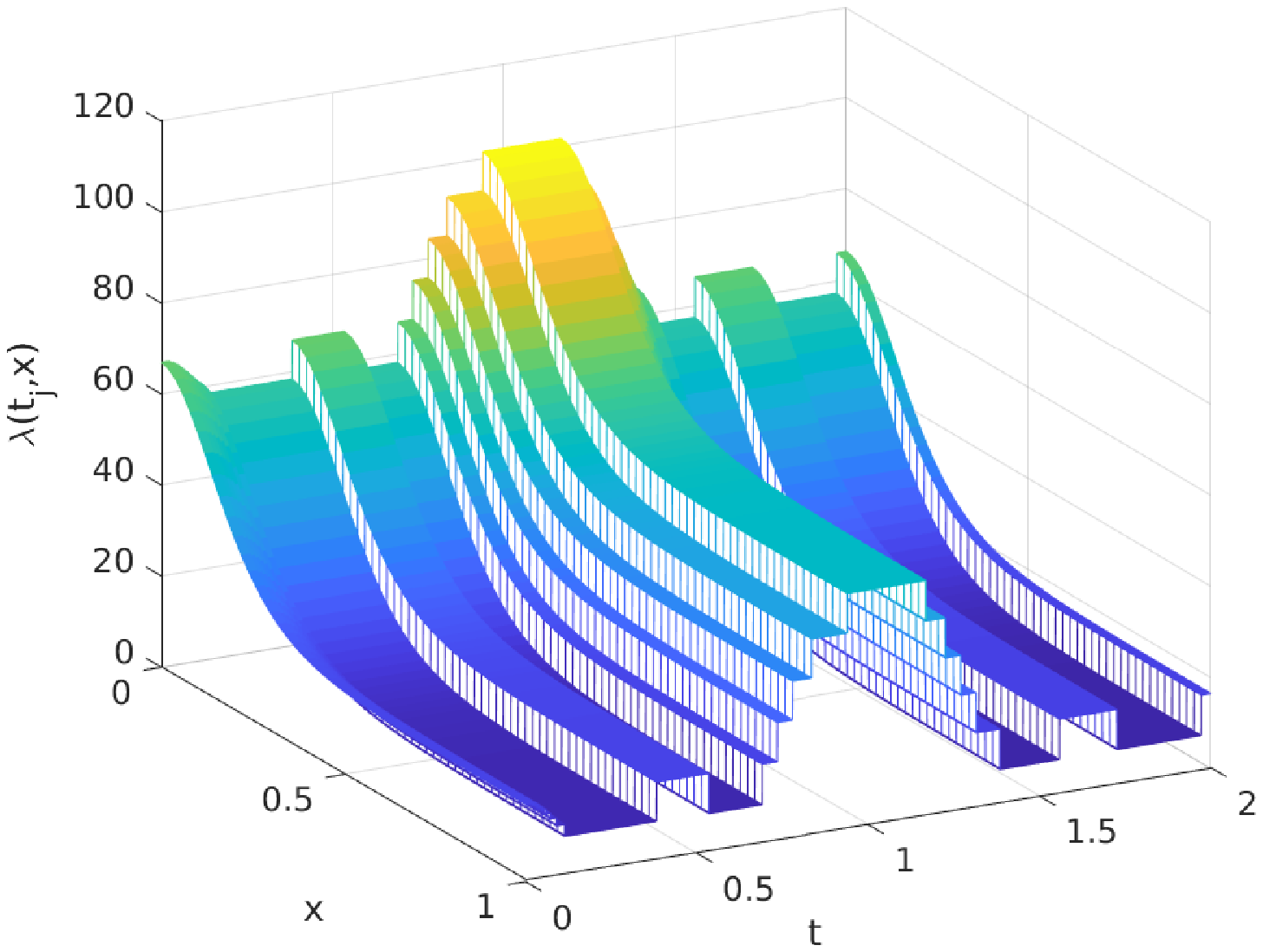}}
\subfigure{\includegraphics[width=0.49\columnwidth]{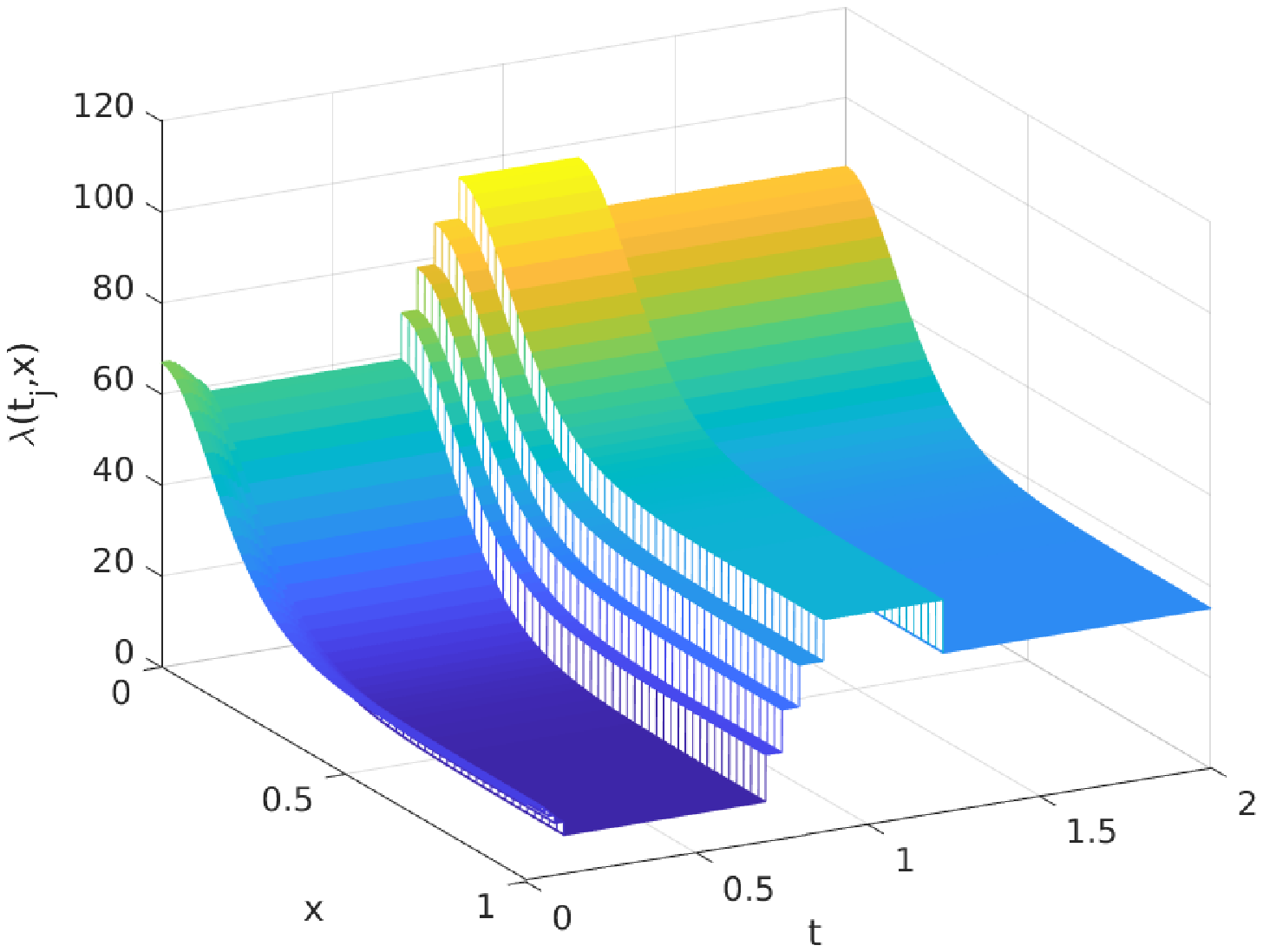}}
\caption{Sampled version of the profile of  the time- and space- varying reaction coefficient \eqref{lambda_ex1}, i.e.  $\lambda(t_j,x)$ for all $\{t_j\}_{j\in \mathbb{N}}$ according to the  static event-triggered gain scheduler  \eqref{triggering_conditionISS_with_backstepping_original} (depicted on the left)    and the  dynamic event-triggered gain scheduler \eqref{triggering_conditionISS_with_backstepping_original_Dynamic}-\eqref{eq:internal_dynamic_variable} (depicted on the right).} 
\label{profile_reaction-terms_sampled}
}
\end{figure*}
\noindent Figure \ref{profile_reaction-terms_sampled} shows the  event-triggered sampled version of the profile of  the time- and space- varying reaction coefficient \eqref{lambda_ex1} for all $t \in [t_j,t_{j+1})$, $j\geq 0$. according to the  static event-triggered gain scheduler  \eqref{triggering_conditionISS_with_backstepping_original} (depicted on the left)    and the  dynamic event-triggered gain scheduler  \eqref{triggering_conditionISS_with_backstepping_original_Dynamic}-\eqref{eq:internal_dynamic_variable} (depicted on the right).
  Hence,  the kernel updating is done on events and aperiodically. One of the main features of this approach is that the kernel of the control  does not need to be computed using the method of successive approximations to solve a PDE kernel  which involves  a time- and space- varying coefficient. As motivated throughout the paper,   it suffices to schedule the kernel in a suitable way and only when needed  while using  a simpler kernel (in some cases admitting closed-form solution). 

Figures \ref{L2_norm} and \ref{control_evolution} show the time-evolution of the $L^2$ norm of the closed-loop system  \eqref{eq:sysparabolic0}-\eqref{IC_parabolic_PDE_u0}, \eqref{lambda_ex1}  and the time-evolution of the  boundary  control \eqref{eq:control_feedback}, respectively,     under static event-triggered gain scheduler \eqref{triggering_conditionISS_with_backstepping_original} (red dashed line) and  dynamic event-triggered gain scheduler \eqref{triggering_conditionISS_with_backstepping_original_Dynamic}-\eqref{eq:internal_dynamic_variable} (blue line).   For both figures, on  the right, there are  zooms of the two curves to illustrate the difference. It can be observed that under the two strategies, the behavior is similar  with same theoretical guarantees. 
\begin{figure*}[h]
\centering{
\subfigure{\includegraphics[width=0.48\columnwidth]{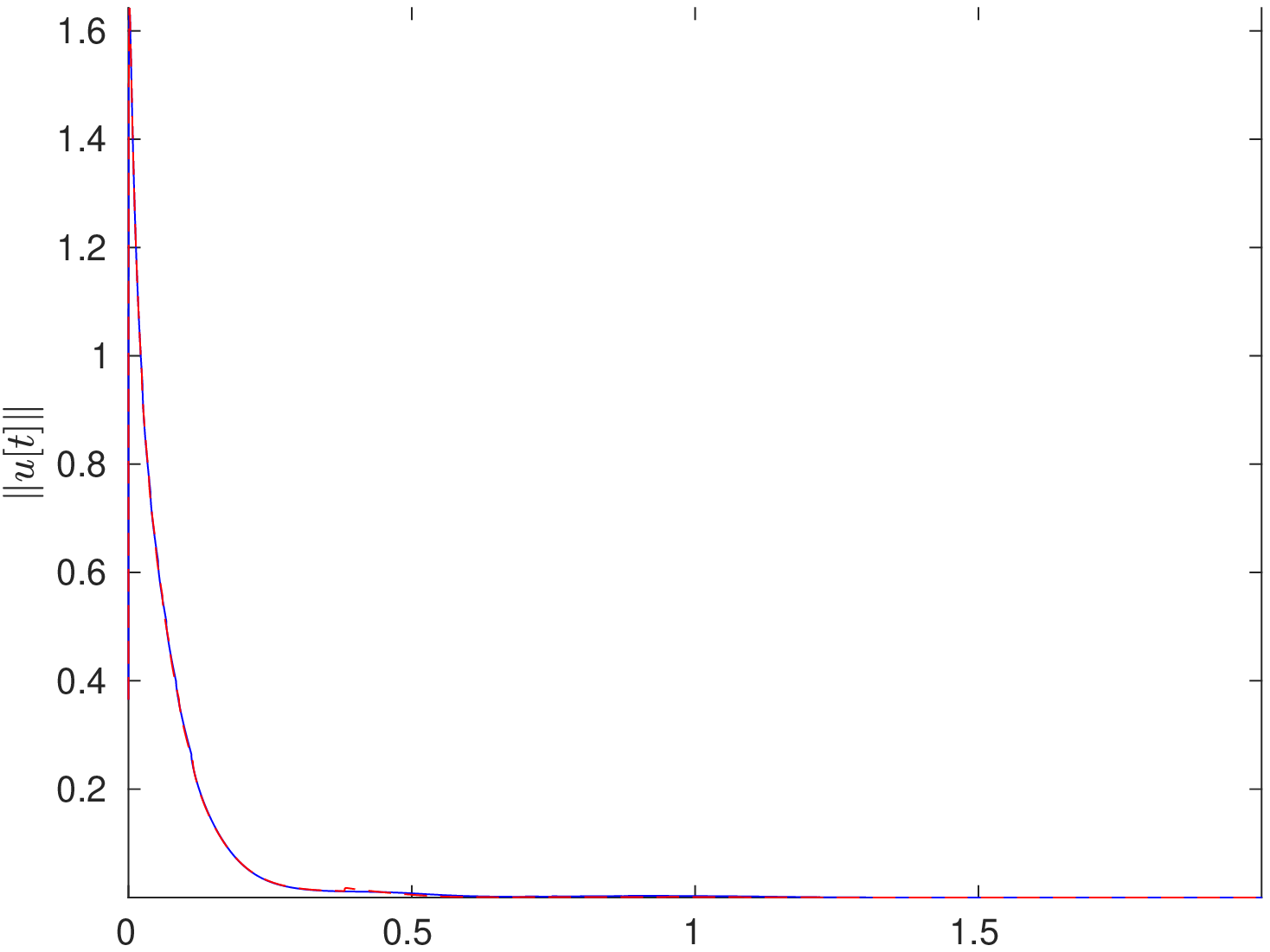} }
\subfigure{\includegraphics[width=0.48\columnwidth]{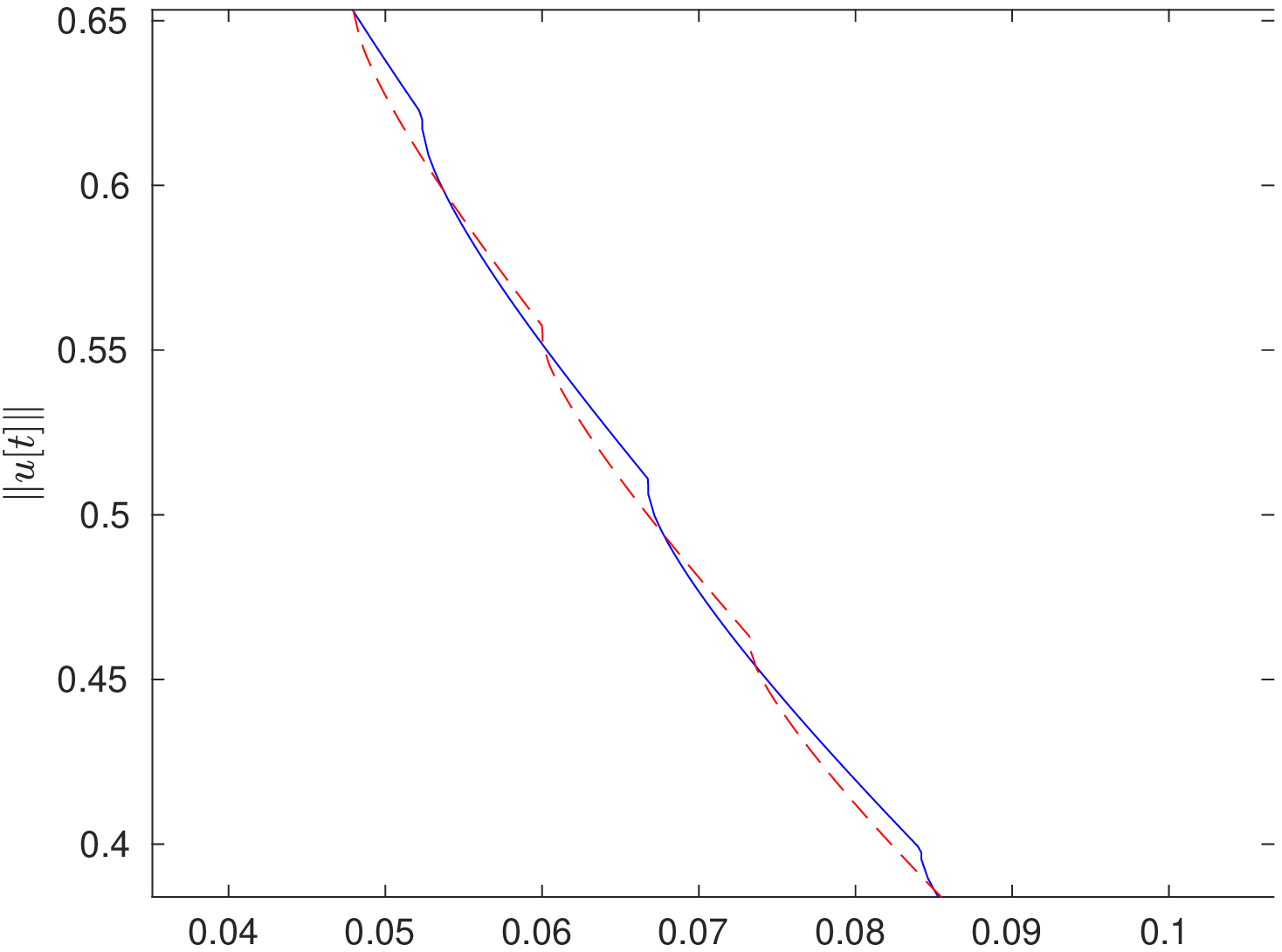} }
\caption{Time-evolution of the $L^2$ norm of the closed-loop system  \eqref{eq:sysparabolic0}-\eqref{IC_parabolic_PDE_u0}, \eqref{lambda_ex1}  with boundary  control    \eqref{eq:control_feedback} under static event-triggered gain scheduler \eqref{triggering_conditionISS_with_backstepping_original} (red dashed  line) and  dynamic event-triggered gain scheduler \eqref{triggering_conditionISS_with_backstepping_original_Dynamic}-\eqref{eq:internal_dynamic_variable} (blue line). On the right, there is a zoom of the two curves to illustrate the difference.} 
\label{L2_norm}
}
\end{figure*}
\begin{figure*}[t]
\centering{
\subfigure{\includegraphics[width=0.48\columnwidth]{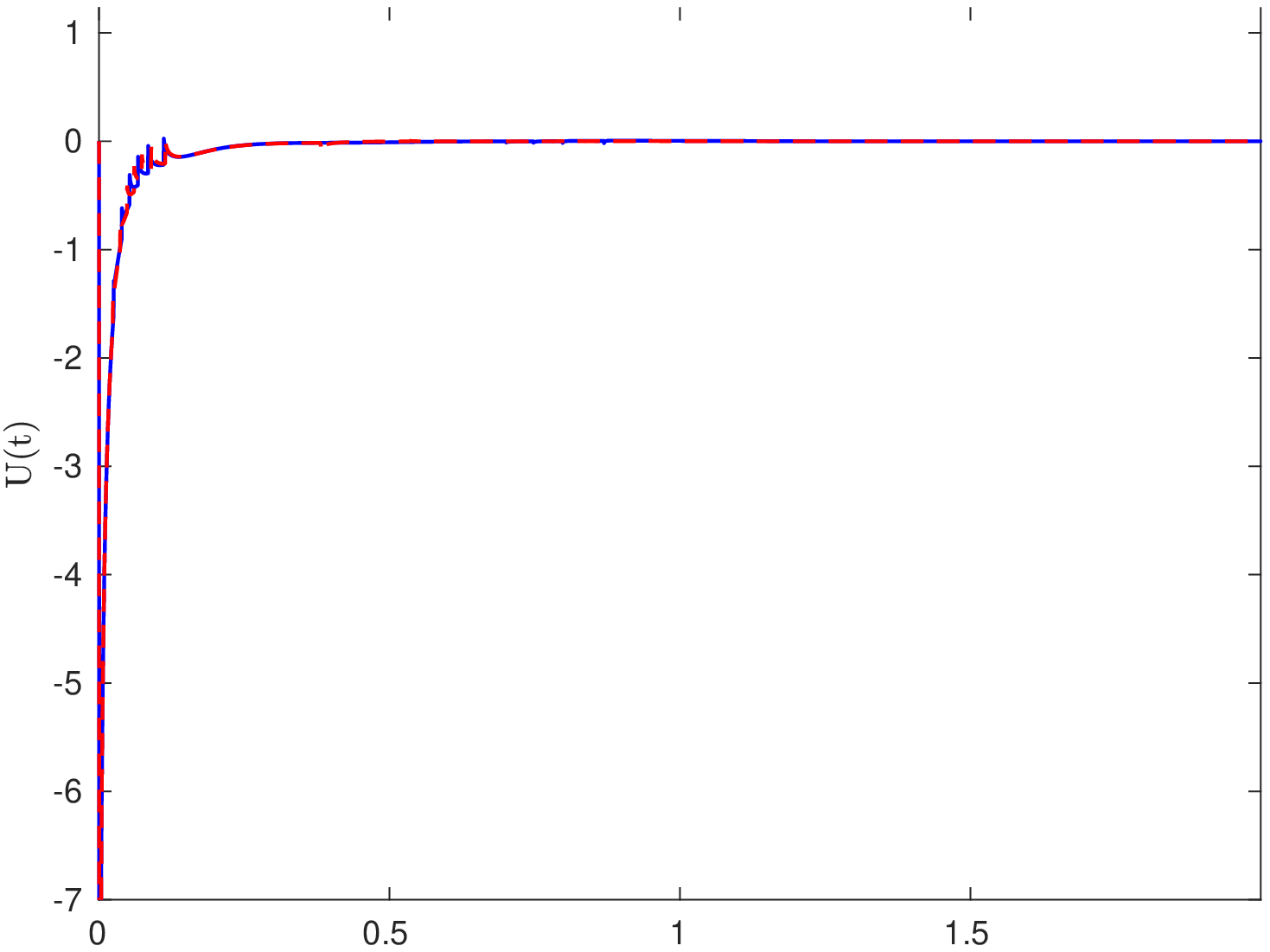} }
\subfigure{\includegraphics[width=0.48\columnwidth]{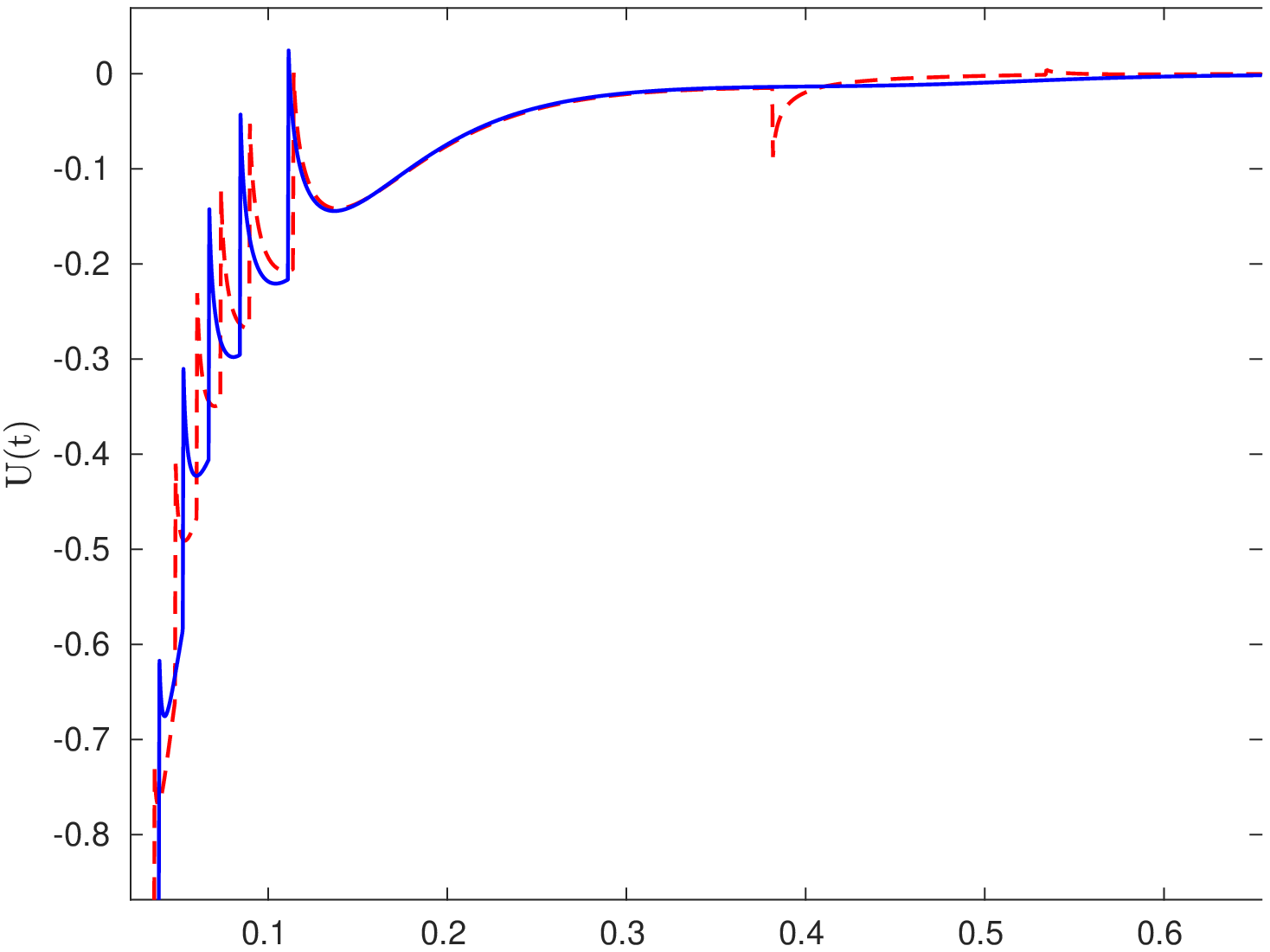} }
\caption{Time-evolution of the boundary control  under static event-triggered gain scheduler \eqref{triggering_conditionISS_with_backstepping_original} (red dashed line) and under  dynamic event-triggered gain scheduler \eqref{triggering_conditionISS_with_backstepping_original_Dynamic}-\eqref{eq:internal_dynamic_variable} (blue line). On the right, there is a zoom of the two curves to illustrate the difference. } 
\label{control_evolution}
}
\end{figure*}
Finally, we run simulations for $100$ different initial conditions given by $u_0(x)= \sqrt{2/n}\sin(\sqrt{n}\pi x) + \sqrt{n}(x -x^2) $,  for $n=1,\ldots,100$ on a frame of $2$s. We compare the static event triggered mechanism with respect to  the dynamic one while computing the inter-execution times  between two triggering times. We compare several  cases by tuning  different parameters.  For all cases, $\eta$ is selected as $\eta =2\mu(1-R)$.  
The mean value of the numbers of events generated under the two strategies is reported in Table \ref{Table_interexecution_times_Mean_valuenumber_events}. The mean value and  coefficient of variation (ratio between the standard deviation and the mean value)
of inter-execution times for both approaches are reported in Tables  \ref{Table_interexecution_times_mean_interexecutions} and \ref{Table_interexecution_times_variability}, respectively.
In addition, Figure \ref{histograms} shows the density of
the inter-execution times  (axis in  logarithmic scale). The red bars correspond to the inter-execution times under  the static event triggered mechanism \eqref{triggering_conditionISS_with_backstepping_original};  whereas  the blue bars     correspond to  the dynamic event triggered mechanism \eqref{triggering_conditionISS_with_backstepping_original_Dynamic}-\eqref{eq:internal_dynamic_variable} resulting in larger inter-execution times.
Therefore, it can be asserted that, as expected, with the dynamic triggering condition one obtains larger inter-execution times and we can reduce the number of events rendering the strategy slightly less conservative. In general, dynamic event-triggered strategies may offer benefits with respect to static strategies as in the framework of even-triggered control (in finite and infinite dimensional settings).

\begin{figure*}[t]
\centering{
\subfigure{\includegraphics[width=0.8\columnwidth]{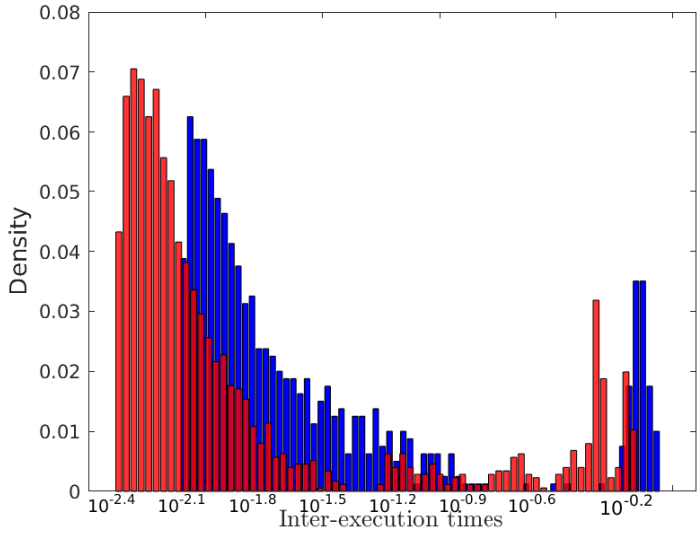} }
\caption{Density of the inter-execution times (axis in logarithmic scale)  computed for 100 different initial conditions given by
$u_0(x)= \sqrt{2/n}\sin(\sqrt{n}\pi x) + \sqrt{n}(x -x^2) $,  for $n=1,\ldots,100$ on a frame of  $2$s. The parameters of the event-triggered strategies are: $R=0.5$, $\eta=9.86 $ and $\theta=0.015$.
The red bars correspond to the density of inter-execution times under  the static event triggered mechanism \eqref{triggering_conditionISS_with_backstepping_original}; whereas  the blue bars  correspond to  the dynamic event triggered mechanism \eqref{triggering_conditionISS_with_backstepping_original_Dynamic}-\eqref{eq:internal_dynamic_variable} resulting in larger inter-execution times.}
\label{histograms}
}
\end{figure*}
\begin{table}[h]
\caption{Mean value of the number of events generated under the  static event-triggered gain scheduler  \eqref{triggering_conditionISS_with_backstepping_original} and under  dynamic event-triggered gains scheduler \eqref{triggering_conditionISS_with_backstepping_original_Dynamic}-\eqref{eq:internal_dynamic_variable}.  \medskip} 
\centering{ 
\begin{tabular}{|c||c|c|c|}
\hline
& $R = 0.15$, $\eta = 16.7$ & $R=0.5$ $\eta =9.86$ \\
\hline  
Static ET                      & 39.93  & 18.58  \\
\hline 
Dynamic ET ($\theta =100$)      &  37.08 & 17.6 \\
 
\hline
Dynamic ET ($\theta =1$)       & 29.8 & 16.02 \\
\hline
Dynamic ET  ($\theta =0.015$)  & 17.01 & 8.99 \\
\hline 
\end{tabular} 
\label{Table_interexecution_times_Mean_valuenumber_events}
}
\end{table}
\begin{table}[h]
\caption{Mean value  of inter-execution times for static event-triggered gain scheduler  \eqref{triggering_conditionISS_with_backstepping_original} and for  dynamic event-triggered gains scheduler \eqref{triggering_conditionISS_with_backstepping_original_Dynamic}-\eqref{eq:internal_dynamic_variable}.  \medskip} 
\centering{ 
\begin{tabular}{|c||c|c|c|}
\hline
& $R = 0.15$, $\eta = 16.7$ & $R=0.5$ $\eta =9.86$ \\
\hline 
Static ET                     & 0.0460 & 0.0738 \\
\hline 
Dynamic ET ($\theta =100$)     &  0.0354 & 0.0521 \\
\hline
Dynamic ET ($\theta =1$)       & 0.0374 &  0.0582 \\ 
\hline

Dynamic ET ($\theta =0.015$)  & 0.0546  & 0.1112 \\ 
\hline 
\end{tabular} 
\label{Table_interexecution_times_mean_interexecutions}
}
\end{table}
\begin{table}[h]
\caption{Coefficient of variation  of inter-execution times for static event-triggered gain scheduler  \eqref{triggering_conditionISS_with_backstepping_original} and for  dynamic event-triggered gains scheduler \eqref{triggering_conditionISS_with_backstepping_original_Dynamic}-\eqref{eq:internal_dynamic_variable}.  \medskip} 
\centering{ 
\begin{tabular}{|c||c|c|c|}
\hline
& $R = 0.15$, $\eta = 16.7$ & $R=0.5$ $\eta =9.86$ \\
\hline 
Static ET                     & 1.9814 & 2.203  \\ 
\hline 
Dynamic ET ($\theta =100$)      & 2.110  & 2.693  \\ 
\hline
Dynamic ET ($\theta =1$)       & 2.6251 &  2.8905  \\
\hline
 
Dynamic ET ($\theta =0.015$)   & 2.3457 & 2.0965 \\
\hline 
\end{tabular} 
\label{Table_interexecution_times_variability}
}
\end{table}
\section{Conclusion}\label{conslusion_and_perspect}

In this paper, we have addressed the problem of exponential stabilization of  a reaction-diffusion PDE   with time- and space- varying reaction coefficient. The boundary control design relies on the backstepping method and the gains are computed/updated on events according to  two event-triggered gain scheduling schemes. 
Two event-triggered strategies are prosed for gain scheduling: static and dynamic. The latter involves a dynamic variable that can be viewed as the filtered value of the static one. It has been observed that under this strategy it is possible to reduce the number of events for the gain scheduling.
We show that under the two proposed event-triggered gain scheduling schemes   Zeno behavior is avoided, which allows to prove  well-posedness  as well as   the exponential stability of the closed-loop system.
 
Our approach   can be seen as an efficient way of kernel computation as it is scheduled aperiodically,  when needed  and relying on the current state information of the closed-loop system. This work constitutes an effort towards the ``robustification" of boundary controllers designed under backstepping method.

In future work, we  expect to   combine this result with event-triggered control strategies for boundary controlled reaction-diffusion PDEs systems recently introduced in \cite{Espitia_Karafyllis_Krstic2019} (which deals with  constant reaction coefficient only). The results in this paper may suggest that the triggering times for  gain scheduling may be synchronized with the time instants for control updating. The control is going to be piecewise constant and not piecewise continuous as in the present work. This would represent a more realistic way of actuation on the PDE system towards digital realizations.

\appendix 

\section{Proof ot Theorem \ref{Theore:notion_of_solution_general_setting}}\label{Proof_Theorem_well_posedness}

\begin{proof}
It suffices to show that there exists $k>0$ such that for each $w_{0} \in L^{2}(0,1)$ and $T>0$ the initial
value problem
\begin{equation}\label{eq:initial_value_problem_general_setting_notion_solution}
\begin{aligned}
\dot{y}[t]+(A+k I) y[t]&=\mathcal{F}(t) y[t]\\
y[0]&=w_{0}
\end{aligned}
\end{equation}
has a unique classical solution on $[0, T]$ in the sense described in  \cite{Pazy1983:book}  where $A: D \rightarrow L^{2}(0,1)$ is the Sturm-Liouville operator defined by the following formula for every $f \in D:$
\begin{equation}
(A f)(x)=-\varepsilon f^{\prime \prime}(x)+c f(x), \text { for } x \in(0,1)
\end{equation}
Notice that any solution $y[t]$ of \eqref{eq:initial_value_problem_general_setting_notion_solution} provides a solution of the initial-boundary value problem \eqref{eq:parabolic_pDE_general_setting_notion_solution} with \eqref{eq:initial_condition_general_seeting_notion_solution} by means of the formula $w[t]=\exp (k t) y[t]$ and any solution $w[t]$ of the initial boundary value problem \eqref{eq:parabolic_pDE_general_setting_notion_solution} with \eqref{eq:initial_condition_general_seeting_notion_solution} provides a solution of the initial value problem \eqref{eq:initial_value_problem_general_setting_notion_solution} by means of the formula $y[t]=\exp (-k t) w[t]$.

Since $-A$ is the infinitesimal generator of a $C_{0}$ semigroup $S(t), t \geq 0$ on $L^{2}(0,1)$ and since for each $t \geq 0$ the operator $\mathcal{F}(t): L^{2}(0,1) \rightarrow L^{2}(0,1)$ is a linear bounded operator for which
there exist constants $\Omega_1, \Omega_2>0$ such that \eqref{eq:constant_bound_1_operatorF} and \eqref{eq:constant_bound_2_operatorF} hold, it follows from Theorem 1.2 on page 184 in \cite{Pazy1983:book} that there exists a unique mild solution $y \in C^{0}\left([0, T] ; L^{2}(0,1)\right)$ of the initial value
problem \eqref{eq:initial_value_problem_general_setting_notion_solution}, i.e.,
\begin{equation}\label{eq:mild_solution_general_setting}
y[t]=\exp (-k t) S(t) w_{0}+\int_{0}^{t} \exp (-k(t-s)) S(t-s) \mathcal{F}(s) y[s] ds, \text { for all } t \in[0, T]
\end{equation}
 Theorem 2.2 on page 4 in \cite{Pazy1983:book} implies the existence of constants $M, \omega>0$ such that
the estimate $\|S(t)\| \leq M \exp (\omega t)$ holds for all $t \geq 0 .$ Exploiting the previous estimate in conjunction with  \eqref{eq:mild_solution_general_setting} and \eqref{eq:constant_bound_1_operatorF}, we get for all $t \in[0, T]:$
\begin{equation}\label{eq:estimate_mild_solution}
\|y[t]\| \leq M \exp (-(k-\omega) t)\left\|w_{0}\right\|+\frac{M c}{k-\omega} \max _{0 \leq s \leq T}(\|y[s]\|), \text { for all } t \in[0, T]
\end{equation}
Selecting $k>0$ so that $\frac{M c}{k-\omega}<1$, estimate \eqref{eq:estimate_mild_solution}  implies the following estimate:
\begin{equation}\label{eq:estimate_mild_interms_intitial_cond}
\max _{0 \leq s \leq T}(\|y[s]\|) \leq\left(1-\frac{M c}{k-\omega}\right)^{-1} M\left\|w_{0}\right\|
\end{equation}
Notice that the mild solution $y \in C^{0}\left([0, T] ; L^{2}(0,1)\right)$ of the initial value problem \eqref{eq:initial_value_problem_general_setting_notion_solution} is a mild solution of the inhomogeneous initial value problem
\begin{equation}
\begin{aligned}
&\dot{y}[t]+(A+k I) y[t]=g(t)\\
&y[0]=w_{0}
\end{aligned}
\end{equation}
with $g(t)=\mathcal{F}(t) y[t]$ for $t \in[0, T]$.  Inequalities \eqref{eq:estimate_mild_interms_intitial_cond}  and \eqref{eq:constant_bound_1_operatorF} imply that $g \in L^{p}\left([0, T] ; L^{2}(0,1)\right)$ for every $p \in[1,+\infty)$.
 Since $-A$ is the infinitesimal generator of an analytic semigroup $S(t)$ on $L^{2}(0,1),$ it follows from Theorem 3.1 on page 110 in \cite{Pazy1983:book} that for every $p \in(1,+\infty),$ the mapping $t \rightarrow y[t]$ is locally Hölder continuous on $(0, T]$ with exponent $\frac{p-1}{p} .$ Using \eqref{eq:constant_bound_1_operatorF},\eqref{eq:constant_bound_2_operatorF} and the fact that $g(t)=\mathcal{F}(t) y[t]$ for $t \in[0, T],$ it follows that for every $p \in(1,+\infty),$ the mapping $t \rightarrow g[t]$ is locally Hölder continuous on $(0, T]$ with exponent $\frac{p-1}{p} .$ By virtue of Corollary 3.3 on page 113 in \cite{Pazy1983:book} we conclude that $y[t]$ is the unique classical solution of \eqref{eq:initial_value_problem_general_setting_notion_solution}.

\end{proof}


\bibliographystyle{siamplain}
\bibliography{ETC_heat}

\begin{thebibliography}{10}

\bibitem{Boskovik_kristick_2001_boun_unstable_heat}
{\sc D.~Boskovic, M.~Krstic, and W.~Liu}, {\em Boundary control of an unstable
  heat equation via measurement of domain-averaged temperature}, IEEE
  Transactions on Automatic Control, 46 (2001), pp.~2022--2028.

\bibitem{Colton1977181}
{\sc D.~Colton}, {\em The solution of initial-boundary value problems for
  parabolic equations by the method of integral operators}, Journal of
  Differential Equations, 26 (1977), pp.~181 -- 190.

\bibitem{Deutcher2017Solvekernels}
{\sc J.~Deutscher and S.~Kerschbaum}, {\em Backstepping control of coupled
  linear parabolic {PIDEs} with spatially-varying coefficients}, IEEE
  Transactions on Automatic Control, 63 (2018), pp.~4218 -- 4233.

\bibitem{Espitia2016_Aut}
{\sc N.~Espitia, A.~Girard, N.~Marchand, and C.~Prieur}, {\em Event-based
  control of linear hyperbolic systems of conservation laws}, Automatica, 70
  (2016), pp.~275--287.

\bibitem{Espitia2018TAC}
{\sc N.~Espitia, A.~Girard, N.~Marchand, and C.~Prieur}, {\em Event-based
  boundary control of a linear 2x2 hyperbolic system via backstepping
  approach}, IEEE Transactions on Automatic Control, 63 (2018), pp.~2686--2693.

\bibitem{Espitia_Karafyllis_Krstic2019}
{\sc N.~Espitia, I.~Karafyllis, and M.~Krstic}, {\em Event-triggered boundary
  control of constant-parameter reaction-diffusion {PDEs:} a small-gain
  approach}, Under review in Automatica. Available at: arXiv:1909.10472,
  (2019).

\bibitem{Espitia2018FTSAutomatica}
{\sc N.~Espitia, A.~Polyakov, D.~Efimov, and W.~Perruquetti}, {\em Boundary
  time-varying feedbacks for fixed-time stabilization of constant-parameter
  reaction-diffusion systems}, Automatica, 103 (2019), pp.~398 -- 407.

\bibitem{Freudenthaler2017}
{\sc G.~Freudenthaler, F.~Göttsch, and T.~Meurer}, {\em Backstepping-based
  extended {Luenberger} observer design for a {Burger-type} pde for multi-agent
  deployment}, in Proc. 20th IFAC World Congress, vol.~50, 2017,
  pp.~6780--6785.

\bibitem{Fridman2012826}
{\sc E.~Fridman and A.~Blighovsky}, {\em Robust sampled-data control of a class
  of semilinear parabolic systems}, Automatica, 48 (2012), pp.~826--836.

\bibitem{girard2014dynamic}
{\sc A.~Girard}, {\em Dynamic triggering mechanisms for event-triggered
  control}, IEEE Transactions on Automatic Control, 60 (2015), pp.~1992--1997.

\bibitem{hante2009modeling}
{\sc F.~Hante, G.~Leugering, and T.~Seidman}, {\em Modeling and analysis of
  modal switching in networked transport systems}, Applied Mathematics and
  Optimization, 59 (2009), pp.~275--292.

\bibitem{event-triger-Heeme-Johan-Tabu}
{\sc W.~Heemels, K.~Johansson, and P.~Tabuada}, {\em An introduction to
  event-triggered and self-triggered control}, in Proceedings of the 51st IEEE
  Conference on Decision and Control, Maui, Hawaii, 2012, pp.~3270--3285.

\bibitem{Hetel2017309}
{\sc L.~Hetel, C.~Fiter, H.~Omran, A.~Seuret, E.~Fridman, J.-P. Richard, and
  S.~Niculescu}, {\em Recent developments on the stability of systems with
  aperiodic sampling: An overview}, Automatica, 76 (2017), pp.~309 -- 335.

\bibitem{JADACHOWSKI2012798}
{\sc L.~Jadachowski, T.~Meurer, and A.~Kugi}, {\em An efficient implementation
  of backstepping observers for time-varying parabolic pdes}, IFAC Proceedings
  Volumes, 45 (2012), pp.~798 -- 803.
\newblock 7th Vienna International Conference on Mathematical Modelling.

\bibitem{JiangSmallGainETC}
{\sc Z.-P. Jiang, T.~Liu, and P.~Zhang}, {\em Event-triggered control of
  nonlinear systems: A small-gain paradigm}, in 13th IEEE International
  Conference on Control Automation (ICCA), 2017, pp.~242--247.

\bibitem{KARAFYLLIS2018226}
{\sc I.~Karafyllis and M.~Krstic}, {\em Sampled-data boundary feedback control
  of {1-D parabolic PDEs}}, Automatica, 87 (2018), pp.~226 -- 237.

\bibitem{Karafyllis2019_book}
{\sc I.~Karafyllis and M.~Krstic}, {\em Input-to-State Stability for PDEs},
  Springer-Verlag, London (Series: Communications and Control Engineering),
  2019.

\bibitem{Kerschbaum_Deutscher2019}
{\sc S.~{Kerschbaum} and J.~{Deutscher}}, {\em Backstepping control of coupled
  linear parabolic {PDEs} with space and time dependent coefficients}, IEEE
  Transactions on Automatic Control,  (2019), pp.~1--1.

\bibitem{krstic2008boundary}
{\sc M.~Krstic and A.~Smyshlyaev}, {\em Boundary control of PDEs: A course on
  backstepping designs}, vol.~16, SIAM, 2008.

\bibitem{POLamare2015}
{\sc P.-O. Lamare, A.~Girard, and C.~Prieur}, {\em Switching rules for
  stabilization of linear systems of conservation laws}, SIAM Journal on
  Control and Optimization, 53 (2015), pp.~1599--1624.

\bibitem{Liberzon2003}
{\sc D.~Liberzon}, {\em Switching in Systems and Control.}, New York: Springer,
  2003.

\bibitem{Liu_ZPJiang2015}
{\sc T.~Liu and Z.-P. Jiang}, {\em A small-gain-approach to robust
  event-triggered control of nonlinear systems}, IEEE Transaction on Automatic
  Control, 60 (2015), pp.~2072--2085.

\bibitem{Logemann2005}
{\sc H.~Logemann, R.~Rebarber, and S.~Townley}, {\em Generalized sampled-data
  stabilization of well-posed linear infinite-dimensional systems}, SIAM
  Journal on Control and Optimization, 44 (2005), pp.~1345--1369.

\bibitem{Meurer2013}
{\sc T.~Meurer}, {\em Control of {Higher Dimensional PDEs}}, Communications and
  Control Engineering, 2013.

\bibitem{MEURER20091182}
{\sc T.~Meurer and A.~Kugi}, {\em Tracking control for boundary controlled
  parabolic {PDEs} with varying parameters: Combining backstepping and
  differential flatness}, Automatica, 45 (2009), pp.~1182--1194.

\bibitem{Orlov_output_coupledparabolicSIAM2017}
{\sc Y.~Orlov, A.~Pisano, A.~Pilloni, and E.~Usai}, {\em Output feedback
  stabilization of coupled reaction-diffusion processes with constant
  parameters}, SIAM Journal on Control and Optimization, 55 (2017),
  pp.~4112--4155.

\bibitem{Pazy1983:book}
{\sc A.~Pazy}, {\em Semigroups of Linear Operators and Applications to Partial
  Differential Equations}, Springer, 1983.

\bibitem{Polyakov2017_heatequ}
{\sc A.~Polyakov, J.-M. Coron, and L.~Rosier}, {\em On boundary finite-time
  feedback control for heat equation}, in 20th IFAC World Congress, Toulouse,
  France, 2017.

\bibitem{Postoyan_Aframework_ETS2014}
{\sc R.~Postoyan, P.~Tabuada, D.~{Nešić}, and A.~Anta}, {\em A framework for
  the event-triggered stabilization of nonlinear systems}, IEEE Transactions on
  Automatic Control, 60 (2015), pp.~982--996.

\bibitem{Prieur2014swithinghyperbolic}
{\sc C.~Prieur, A.~Girard, and E.~Witrant}, {\em Stability of switched linear
  hyperbolic systems by {Lyapunov} techniques}, IEEE Transactions on Automatic
  Control, 59 (2014), pp.~2196--2202.

\bibitem{Selivanov_FridmanAuto}
{\sc A.~Selivanov and E.~Fridman}, {\em Distributed event-triggered control of
  transport-reaction systems}, Automatica, 68 (2016), pp.~344--351.

\bibitem{Smyshlyaev-Krstic2004}
{\sc A.~Smyshlyaev and M.~Krstic}, {\em Closed-form boundary state feedbacks
  for a class of 1-d partial integro-differential equations}, IEEE Transactions
  on Automatic Control,, 49 (2004), pp.~2185--2202.

\bibitem{SMYSHLYAEV20051601}
{\sc A.~Smyshlyaev and M.~Krstic}, {\em On control design for pdes with
  space-dependent diffusivity or time-dependent reactivity}, Automatica, 41
  (2005), pp.~1601 -- 1608.

\bibitem{Strauss2008}
{\sc W.~Strauss}, {\em Partial Differential Equations}, {Wiley}, 2nd~ed., 2008.

\bibitem{tabuada2007event}
{\sc P.~Tabuada}, {\em Event-triggered real-time scheduling of stabilizing
  control tasks}, IEEE Transactions on Automatic Control, 52 (2007),
  pp.~1680--1685.

\bibitem{Tan2009}
{\sc Y.~Tan, E.~{Trélat}, Y.~Chitour, and D.~{Nešić}}, {\em Dynamic
  practical stabilization of sampled-data linear distributed parameter
  systems}, in IEEE 48th Conference on Decision and Control (CDC), 2009,
  pp.~5508--5513.

\bibitem{Vaszques2017parabolic75}
{\sc R.~Vazquez and M.~Krstic}, {\em Boundary control of coupled
  reaction-advection-diffusion systems with spatially-varying coefficients},
  IEEE Transactions on Automatic Control, 62 (2017), pp.~2026--2033.

\bibitem{VasquezTrelatCoron}
{\sc R.~Vazquez, E.~{Trélat}, and J.-M. Coron}, {\em Control for fast and
  stable laminar-to-high-reynolds-numbers transfer in a {2D} {Navier-Stokes}
  channel flow}, Discrete \& Continuous Dynamical Systems - B, 10 (2008),
  pp.~925--956.

\end{thebibliography}
\end{document}